\documentclass[preprint,onefignum,onetabnum]{siamart171218}

\usepackage{lipsum}
\usepackage{amsfonts}
\usepackage{graphicx}
\usepackage{epstopdf}
\usepackage{algorithmic}
\ifpdf
  \DeclareGraphicsExtensions{.eps,.pdf,.png,.jpg}
\else
  \DeclareGraphicsExtensions{.eps}
\fi

\newsiamremark{remark}{Remark}
\newsiamremark{hypothesis}{Hypothesis}
\crefname{hypothesis}{Hypothesis}{Hypotheses}
\newsiamthm{claim}{Claim}

\headers{Augmented Backward-Corrected PSI for DLRT}{J. Kusch, S. Schotthöfer, and A. Walter}

\title{An Augmented Backward-Corrected Projector Splitting Integrator for Dynamical Low-Rank Training\thanks{Submitted to the editors February 05, 2025.
\funding{
Alexandra Walter is funded by the Helmholtz Information \& Data Science School for Health (HIDSS4Health). A 3-month research stay of AW at the Norwegian University of Life Sciences (NMBU) was funded by the Norway Exchange Program of the Helmholtz Information and Data Science Academy (HIDA).\\
The work of Steffen Schotthöfer is sponsored by the Applied Mathematics Program at the Office of Advanced Scientific Computing Research, U.S. Department of Energy, and performed at the Oak Ridge National Laboratory, which is managed by UT-Battelle, LLC under Contract No. DE-AC05-00OR22725 with the U.S. Department of Energy. The United States Government retains and the publisher, by accepting the article for publication, acknowledges that the United States Government retains a non-exclusive, paid-up, irrevocable, world-wide license to publish or reproduce the published form of this manuscript, or allow others to do so, for United States Government purposes. The Department of Energy will provide public access to these results of federally sponsored research in accordance with the DOE Public Access Plan (http://energy.gov/downloads/doe-public-access-plan).
}}}

\author{Jonas Kusch
    \thanks{Norwegian University of Life Sciences, \AA s, Norway (\email{jonas.kusch@nmbu.no}).}
\and Steffen Schotthöfer
    \thanks{Computer Science and Mathematics Division, 
  Oak Ridge National Laboratory, 
  Oak Ridge, TN 37831 USA} (\email{schotthofers@ornl.gov}).
\and Alexandra Walter
    \thanks{Corresponding author: Scientific Computing Center, Karlsruhe Institute of Technology (KIT), Karlsruhe, Germany; Division of Medical Physics in Radiation Oncology, German Cancer Research Center (DKFZ); Heidelberg, Germany; Heidelberg Institute of Radiation Oncology (HIRO) \& National Center for Radiation Research in Oncology (NCRO), Heidelberg/Dresden, Germany (\email{alexandra.walter@kit.edu}).}
}

\usepackage{amsopn}

\makeatletter
\newcommand*{\addFileDependency}[1]{
  \typeout{(#1)}
  \@addtofilelist{#1}
  \IfFileExists{#1}{}{\typeout{No file #1.}}
}
\makeatother

\newcommand*{\myexternaldocument}[1]{%
    \externaldocument{#1}%
    \addFileDependency{#1.tex}%
    \addFileDependency{#1.aux}%
}

\usepackage{amsmath}
\usepackage{amssymb}
\usepackage{mathtools}
\usepackage{enumitem}
\usepackage{booktabs}
\ifpdf
\hypersetup{
  pdftitle={An Augmented Backward-Corrected Projector Splitting Integrator for Dynamical Low Rank Training},
  pdfauthor={Jonas Kusch, Steffen Schotth\"ofer, and Alexandra Walter}
}
\fi

\newcommand{\stochgrad}{\nabla\ell}

\myexternaldocument{ex_supplement}

\begin{document}

\maketitle

\begin{abstract}
Layer factorization has emerged as a widely used technique for training memory-efficient neural networks. However, layer factorization methods face several challenges, particularly a lack of robustness during the training process. To overcome this limitation, dynamical low-rank training methods have been developed, utilizing robust time integration techniques for low-rank matrix differential equations. Although these approaches facilitate efficient training, they still depend on computationally intensive QR and singular value decompositions of matrices with small rank. In this work, we introduce a novel low-rank training method that reduces the number of required QR decompositions. Our approach integrates an augmentation step into a projector-splitting scheme, ensuring convergence to a locally optimal solution. We provide a rigorous theoretical analysis of the proposed method and demonstrate its effectiveness across multiple benchmarks.
\end{abstract}

\begin{keywords}
  Dynamical Low-Rank Training, Dynamical Low-Rank Approximation, Neural Network Training, Projector Splitting Integrator
\end{keywords}

\begin{AMS}
  68T07, 49Q12, 65L05, 65L20, 65L70
\end{AMS}

\section{Introduction}

Machine learning models are continually advancing in their ability to tackle complex tasks, such as segmenting organs at risk and target volumes on CT scans for radiation therapy \cite{weissmann2023deep, d2024totalsegmentator}, providing language-based information and assistance \cite{brown2020language}, or generating images \cite{hatamizadeh2025diffit, khader2023denoising}. Along with this growing complexity and capability, the number of parameters - including depth, width, and feature channels of artificial neural networks - has increased tremendously in recent years \cite{ANN_alzubaidi2021review, ANN_zhao2024review}. In addition to advancements in processing units, managing these large parameter sets relies on various model compression techniques \cite{ANN_compression_li2023model}. These techniques leverage the observation that models are commonly over-parameterized \cite{frankle2018lottery, bah2022learning, denil2013predicting} which has been exploited since the 1990s \cite{lecun1989optimal, hassibi1992second}. The most prominent compression techniques are sparsification \cite{guo2016dynamic, molchanov2017pruning, he2017channel,khalitovchordmixer}, quantization \cite{wu2016quantized, courbariaux2016binarized}, and layer factorization. The latter has recently gained considerable attention, especially for fine-tuning tasks \cite{hu2021lora, valipour2023dylora, zhang2023adalora, hayou2024lora, zhao2024galore, lialin2023relora, schotthöfer2024geolorageometricintegrationparameter}, however also for pre-training \cite{wang2021pufferfish, khodak2021initialization, schotthofer2022low, schotthöfer2024federateddynamicallowranktraining, zangrando2023rank, zhao2024galore}. While some of these methods reduce network size after training, others can compress the network during training. Approaches in the latter category avoid the computationally expensive training of a full-scale network, significantly reducing both computational costs and memory requirements. While common training strategies that compress the model during training often lack guarantees of reaching a local optimum, a family of training methods based on the theory of dynamical low-rank approximation \cite{koch2007dynamical} overcomes this limitation, as these methods are specifically designed to satisfy local optimality conditions \cite{schotthofer2022low,zangrando2024geometryawaretrainingfactorizedlayers,schotthöfer2024federateddynamicallowranktraining,schotthöfer2024geolorageometricintegrationparameter}.

In this dynamical low-rank training \cite{schotthofer2022low}, the weights of a neural network are compressed by restricting them to the manifold of low-rank matrices. Therefore, during training, instead of storing and updating large weight matrices, only their factorized form, i.e., small, narrow matrices, are required. The strategy to efficiently and robustly train these factorized matrices is to reformulate the training process as a gradient flow and to evolve the resulting matrix ordinary differential equations through low-rank time integrators developed in the context of dynamical low-rank approximation (DLRA). DLRA, which has been established in \cite{koch2007dynamical}, has been used for various problems in scientific computing, including quantum mechanics \cite{haegeman2016unifying} and kinetic problems \cite{einkemmer2018low}. Various integrators have been proposed in the literature to solve the time evolution equations in dynamical low-rank approximation. The most frequently used integrators are projector--splitting integrators (PSI) \cite{lubich2014projector,kieri2016discretized,hochbruck2023rank} and basis-update \& Galerkin (BUG) integrators \cite{ceruti2022unconventional,ceruti2022rank,ceruti2024parallel}.
In fields like quantum physics, the projector splitting integrator is widely used \cite{haegeman2016unifying}. For a gyrokinetic model, the PSI shows greater efficiency, and improved stability for larger time steps with comparable accuracy when compared to other suitable integrators \cite{einkemmer2024accelerating}. However, most research on kinetic equations and dynamical low-rank training has primarily focused on basis-update \& Galerkin integrators. One primary reason for this development is that the PSI requires solving a subproblem that evolves the underlying dynamics backward in time. In dynamical low-rank training, this can increase the loss, destroying the convergence to optimal weights. 

In this work, we analyze the PSI in the context of neural network training, resolving the issue of the backward-in-time subproblem by deriving a novel augmented backward-corrected version of the PSI that is shown to converge to a locally optimal point. The novel method introduces an augmentation step in the backward-corrected PSI of \cite{bachmayr2021existence}, allowing for rank-adaptivity while ensuring descent and local convergence guarantees. Moreover, the construction reduces the number of QR decompositions required in every training step from two to one.

The paper is structured as follows: After the introduction, we provide a background on low-rank neural networks and dynamical low-rank approximation in \cref{sec:background}. In \cref{sec:DLRT}, we review the use of DLRA for neural network training and discuss projector--splitting time integration methods to train neural networks, emphasizing potential weaknesses. We present our modification to the PSI in \cref{sec:abcPSI} and provide the derivation of the robust error bound as well as local convergence properties in \cref{sec:numan}. Lastly, we compare different projector--splitting integrators used to train low-rank neural networks for the MNIST dataset and to fine-tune a vision transformer in \cref{sec:numexp}.

\section{Background and Notation}\label{sec:background}

\subsection{Neural Network Training}
An artificial feed-forward neural network\\ $\mathcal{N}(x)$ is a recursive composition of affine and non-linear functions. Given non-linear activation functions $\sigma_i$, a feed-forward neural network reads
\begin{align*} 
\mathcal{N}(x) = \sigma_L(W^L a^{L-1}(x) + b^L) \, ,
\end{align*} 
where $a^{L-1}(x)$ is defined recursively by
\begin{align*}
    a^0(x) &= x \in \mathbb{R}^{n_0} \, ,\\
    a^l(x) &= \sigma_l(W^{l} a^{l-1}(x) + b^{l}) \in \mathbb{R}^{n^l}, & l &= 1, \dots, L \,.
\end{align*} 
Here, $W_{l} \in \mathbb{R}^{n^{l} \times n^{l-1}}$ are called the weight matrices which are collected as $\mathcal{W} = \{W^1, \dots, W^L\} \in \mathbb{R}^p$. Moreover, $b^l\in \mathbb{R}^{n_l}$ are called bias vectors. 
Since this work focuses on weight matrices $\mathcal{W}$ only, biases will be omitted in the following. 

\textit{Training} a neural network $\mathcal{N}(x)$ is the minimization of a cost function \\ $\mathcal{L}(\mathcal{W}; \mathcal{N}(\mathcal X), \mathcal Y)$ with respect to  the weights $\mathcal{W}$ for a given dataset $\mathcal X = \{x_1, \dots x_s\}$ with the corresponding exact labels $\mathcal Y=\{y_1, \dots, y_s \}$. Here, $\mathcal{L}(\mathcal{W}; \mathcal{N}(\mathcal X), \mathcal Y)$ denotes the evaluation of the cost function on the entire dataset $\mathcal X$. Hence, we aim to determine optimal weights $\mathcal{W}_{\star}$ such that 
\begin{align*}
    \mathcal{W}_{\star} = \text{argmin}_{\mathcal{W}}\mathcal{L}(\mathcal{W}; \mathcal{N}(\mathcal X), \mathcal Y)\,.
\end{align*}
For computational efficiency, the full cost function is replaced by a batch evaluation. Given the batch $\mathcal{X}_{\xi} \subset \mathcal X$ with corresponding exact labels ${\mathcal{Y}}_{\xi}$, where ${\mathcal{X}}_{\xi} = \{x_1^{(\xi)},\cdots, x_b^{(\xi)}\}$, ${\mathcal{Y}}_{\xi} = \{y_1^{(\xi)},\cdots, y_b^{(\xi)}\}$, and $b\ll s$, the loss function on the batch ${\mathcal{X}}_{\xi}$ is given by 
\begin{align*}
    \ell(\mathcal{W}; \mathcal{N}(\mathcal{X}_\xi), \mathcal{Y}_\xi) := \frac{1}{b}\sum_{i=1}^b\mathcal{L}(\mathcal{W}; \mathcal{N}(x_i^{(\xi)}), y_i^{(\xi)})\,.
\end{align*}
The elements in ${\mathcal{X}}_\xi$ are then changed in each iteration of the training method to cover the entire training set $\mathcal X$ after a sufficient amount of iterations. This batch evaluation introduces a stochastic influence. Since the pair ${\mathcal{X}}_{\xi},{\mathcal{Y}}_{\xi}$ is drawn from the distribution of the training data in $\mathcal X$, the batch loss fulfills
\begin{align*}
    \mathbb{E}_\xi[ \ell(\mathcal{W}; \mathcal{N}(\mathcal{X}_\xi), \mathcal{Y}_\xi)] = \mathcal{L}(\mathcal{W}; \mathcal{N}(\mathcal X), \mathcal Y)\,.
\end{align*}
Batch evaluation naturally leads to the use of stochastic gradient descent as the optimizer. Specifically, for all weight matrices $W\in\mathcal{W}$, we apply the iterative scheme
\begin{align*} 
W_{k+1} = W_k - h \nabla_{W} \ell(\mathcal W_k; \mathcal{N}(\mathcal{X}_\xi), \mathcal{Y}_\xi) \, ,
\end{align*} 
where $k$ is the training iteration, $h$ the learning rate and training is initialized with $W_0$. 
In expectation, for the stochastic gradient, we have
\begin{align*}
  \mathbb{E}_\xi[\nabla_{W}\ell(\mathcal W; \mathcal{N}(\mathcal X_\xi), \mathcal Y_\xi)] = \nabla_W\mathcal{L}(\mathcal{W}; \mathcal{N}(\mathcal X), \mathcal Y).
\end{align*}

In the following, we make several simplifications which do not lead to a loss of generality and serve the purpose of allowing an efficient presentation.
We abbreviate $\nabla_{W}\ell$ as $\nabla\ell$ and omit the dependence of the gradient on the neural network and labels. Moreover, we assume a network with a single layer; that is, we replace $\mathcal W$ with $W$ inside the loss. We remark that the methodological results presented in \cref{sec:abcPSI} can be extended to a multi-layer network using, e.g., Proposition 1 of \cite{schotthöfer2024geolorageometricintegrationparameter}. 
Then, for a given weight $W\in\mathbb{R}^{m\times n}$, the training dynamics of stochastic descent methods are governed by the stochastic gradient-flow
\begin{align}\label{eq:gradflow}
   \dot{W}(t) = - \nabla \mathcal{\ell}(W(t)), \qquad  W(t=0)=W_0,
\end{align}
where the dot denotes the time derivative. Note that the steepest-descent update corresponds to an explicit Euler time discretization of the gradient flow.

\subsection{Dynamical Low-Rank Approximation}
State-of-the-art neural\\ networks are often massively over-parametrized, i.e., have by orders of magnitude more weights than training data, often expressed by high-dimensional weight matrices $W$. The price to pay is excessive memory and compute cost to train neural networks with stochastic gradient descent. 

Dynamical Low-Rank Approximation (DLRA) has been proposed as a model order reduction technique for matrix ordinary differential equations in~\cite{koch2007dynamical}. The goal is to efficiently determine the true solution $W(t) \in \mathbb{R}^{m \times n}$ of a differential equation such as \eqref{eq:gradflow}, or, more generally, $\dot W(t) = \mathcal{F}(W(t))$, where $\mathcal{F}$ is an arbitrary right-hand side. As our notation suggests, the solution $W(t)$ can be the weight matrix when choosing a single-layered neural network. To reduce computational costs and memory requirements, we aim to approximate $W(t)$ by a low-rank matrix $Y(t) \in \mathbb{R}^{m \times n}$ such that $\| W(t) - Y(t) \|$ is sufficiently small for all times $t$. Omitting dependency on time, a rank $r$ approximation can then be written as $Y = USV^\top\in\mathcal{M}_r\subset \mathbb{R}^{m \times n}$, where the manifold of rank $r$ matrices is denoted by $\mathcal{M}_r$. Here, $U \in \mathbb{R}^{m \times r}$, $S \in \mathbb{R}^{r \times r}$, and $V \in \mathbb{R}^{n \times r}$, where the columns of $U$ and $V$ are orthonormal. This reduces the number of entries from $nm$ for the full-rank matrix $W$ to $(m + n)r + r^2$ for its low-rank approximation $Y$. If the rank $r \ll \min\{ m, n\}$, the memory footprint of the approximation is negligible compared to its full-rank counterpart. When evolving $Y(t)$ in time, one needs to ensure that $Y(t)\in \mathcal{M}_r$ at all times while ensuring that the distance to the full-rank solution is as small as possible. Following~\cite{koch2007dynamical}, this is achieved by solving
\begin{align*} 
\| \dot W(t) - \dot Y(t) \| = \min \quad \text{s.t.} \quad \dot Y(t) \in \mathcal{T}_{Y(t)} \mathcal{M}_r \, ,
\end{align*} 
where $\mathcal{T}_Z \mathcal{M}_r$ denotes the tangent space of $\mathcal{M}_r$ at $Z$ and $\Vert \cdot \Vert$ denotes the Frobenius norm. The evolution of $Y$ along the tangent space ensures that $Y$ stays within the low-rank manifold $\mathcal{M}_r$.
Using the product rule and the factorization $Y = USV^\top$, we obtain
\begin{align*} 
\dot Y = \dot U SV^\top + U \dot S V^\top + US \dot V^\top \, .
\end{align*} 
With this, and using the Gauge conditions $U^\top \dot U = 0$ and $V^\top \dot V = 0$, which ensure the uniqueness of the low-rank representation, evolution equations for the low-rank factors $U$, $S$, and $V$ can be derived \cite{koch2007dynamical}. These equations solve the problem
\begin{align}\label{eq:PF}
    \dot Y(t) = P(Y(t))\mathcal{F}(Y(t))\,,
\end{align}
where for $Z=USV^{\top}$ with $U$ and $V$ having orthonormal columns, $P(Z)$ is the projector onto the tangent space of $\mathcal{M}_r$ at $Z$ which takes the form $P(Z)G = UU^{\top}G - UU^{\top}GVV^{\top} + GVV^{\top}$ for general $G\in\mathbb{R}^{m\times n}$.
A core difficulty when solving \eqref{eq:PF} is that the projector $P$ has a prohibitively large Lipschitz constant \cite[Lemma~4.2]{koch2007dynamical}, that tends to infinity as the smallest singular value of $S$ tends to zero. Geometrically speaking, the condition number of $S$ determines the curvature of $\mathcal{M}_r$ at $Y$, which leads to a prohibitively small time step size to evolve the solution with a conventional time integration method. To address this issue, time integration schemes that are robust to this curvature have been proposed in, e.g., \cite{lubich2014projector, ceruti2022unconventional, ceruti2022rank, ceruti2024robust,ceruti2024parallel,kusch2024second}. In these schemes, the evolution of low-rank factors is restricted to flat subspaces in the low-rank manifold, namely the submanifolds
\begin{align*}
    \mathcal{M}_{K} =\,& \{ KV_0^{\top} | K \in  \mathbb{R}^{m \times r}, \text{and } V_0\in  \mathbb{R}^{n \times r} \text{ fixed}\}\,, \\
    \mathcal{M}_{S} =\,& \{ U_1SV_0^{\top} | S \in  \mathbb{R}^{r \times r}, U_1\in  \mathbb{R}^{m \times r} \text{, fixed and } V_0\in  \mathbb{R}^{n \times r} \text{ fixed}\}\,,\\
    \mathcal{M}_{L} =\,& \{ U_1 L^{\top} | L \in  \mathbb{R}^{n \times r},  \text{and } U_1\in  \mathbb{R}^{m \times r} \text{ fixed}\}\,,
\end{align*}
which exhibit a moderate curvature compared to $\mathcal{M}_r$.

\section{Dynamical Low-Rank Approximation for Neural Network Training}\label{sec:DLRT}

In standard neural network training, the full weight matrix $W$ of size $nm$ is updated until a critical point $W_{\star}$ is approximately reached, where $\mathbb{E}[\nabla \mathcal{\ell}(W_{\star})]=0$. As discussed previously, a key drawback of modern neural network architectures is the large size of weight matrices, which leads to high memory and computational costs during training and prediction. Low-rank training offers a popular solution for reducing network size by training factorized low-rank weights instead of their full-rank, memory-intensive analogs. To achieve this, a constraint is added to the optimization problem, requiring the solution to lie on the manifold of low-rank matrices. In this case, optimality in $\mathbb{R}^{m\times n}$ is commonly not possible. Instead, for a low-rank weight $Y\in\mathcal{M}_r$, the optimality criterion needs to be relaxed to $\mathbb{E}[P(Y_{\star})\nabla \mathcal{\ell}(Y_{\star})]=0$, where again $P(Z)$ is the projection onto the tangent space of $\mathcal{M}_r$ at $Z\in\mathcal{M}_r$, see, e.g., \cite[Theorem~3.4]{sato2021riemannian}. As shown in \cite[Section~3]{schotthöfer2024geolorageometricintegrationparameter}, standard training methods can fail to converge to such an optimum. Instead, new training methods that follow the modified gradient flow problem
\begin{align}\label{eq:grad_flow_proj}
    \dot Y(t) = -P(Y(t))\nabla\mathcal{\ell}(Y(t))
\end{align}
need to be constructed. This problem resembles the projected flow of DLRA \eqref{eq:PF}, which is highly stiff. Therefore, novel training methods that are robust to this stiffness need to be developed, following the principles of robust time integration methods for DLRA.

The goal of dynamical low-rank training (DLRT) \cite{schotthofer2022low} is to develop training methods that train a low-rank weight $Y = USV^{\top}$ by solving the projected gradient flow equation \eqref{eq:grad_flow_proj} while being robust to the curvature of the low-rank manifold. In the following, we review different integrators for DLRA and discuss their applicability for DLRT. To limit the introduction of new variables, we will recycle variable names when their meaning directly follows from the context in which they are used. Commonly, the full-rank weight is denoted as $W$, and low-rank approximations are denoted as $Y = USV^{\top}$ for different integrators.

\subsection{Basis-update and Galerkin Integrator}

The perhaps most frequently used class of integrators are basis-update \& Galerkin (BUG) integrators \cite{ceruti2022unconventional,ceruti2022rank,ceruti2024parallel}. These integrators approximate the projected gradient flow \eqref{eq:grad_flow_proj} robustly, even in the presence of small singular values, i.e., when $S$ is ill-conditioned. In the fixed-rank BUG integrator \cite{ceruti2022unconventional}, $U$ and $V$ are updated in parallel, followed by the update of $S$. Given our stochastic gradient-flow \eqref{eq:gradflow} for a single-layer neural network, the integrator evolves the factorized low-rank approximation $Y(t_0) = U_0S_0V_0^{\top}$ from time $t_0$ to $t_1 = t_0 + h$ according to
\begin{subequations}\label{eq:dlracontBUG}
 \begin{align}
\dot{K}(t) =& \, -\nabla \mathcal{\ell}(K(t)V_0^{\top})V_0 &&\textrm{with } K(t_0) = U_0S_0 \, , \label{eq:kstepcontBUG}\\
\dot{L}(t) =& \, -\nabla \mathcal{\ell}(U_0 L(t)^{\top})^{\top}U_0 &&\textrm{with }L(t_0) = V_0 S_0^\top \, , \label{eq:lstepcontBUG}\\ 
\dot{S}(t) = & \, -U_1^{\top}\nabla \mathcal{\ell}(U_1S(t)V_1^{\top})V_1 &&\textrm{with } S(t_0) = U_1^\top U_0 S_0 V_0^\top V_1 \, , \label{eq:sstepcontBUG}
\end{align}
\end{subequations}
where the orthonormal $U_1$ and $V_1$ are determined by a QR factorization such that $K(t_1) = U_1R_1\in\mathbb{R}^{m\times r}$ and $L(t_1)=V_1R_2\in\mathbb{R}^{n\times r}$. The factorized solution at $t_1$ is then given by $Y(t_1) = U_1 S_1 V_1^\top$, where $S_1 = S(t_1)$. This process is repeated until a desired end time $t_{\mathrm{end}}$ is reached.
While this integrator requires a predefined rank $r$ as input, a rank-adaptive version, commonly called the augmented BUG integrator, has been proposed in \cite{ceruti2022rank}. This integrator changes rank $r$ over time while retaining robustness and other favorable properties of the original fixed-rank integrator. A parallel BUG integrator has been proposed in \cite{ceruti2024parallel}, which updates all factors in parallel.

Due to its ability to adapt the rank along with its strong theoretical guarantees, training low-rank neural networks with DLRT has centered on BUG integrators. The augmented BUG integrator \cite{ceruti2022rank}, in particular, has been applied to train both matrix-valued \cite{schotthofer2022low} and tensor-valued weights \cite{zangrando2024geometryawaretrainingfactorizedlayers}. Recently, the parallel BUG integrator \cite{ceruti2024parallel} was introduced in \cite{schotthöfer2024geolorageometricintegrationparameter} for low-rank fine-tuning. The authors demonstrate that these integrators can compress weights significantly while nearly preserving the network’s accuracy. Additionally, these training methods have been adapted for stochastic gradient flows, ensuring the method's robustness and guaranteeing the descent of the loss function \cite{arsen_stachastic_gradient_descent}.

\subsection{Projector Splitting Integrator (PSI)}\label{sec_psi}

Another well-known example of a robust integrator is the Projector Splitting Integrator (PSI), which has been proposed in \cite{lubich2014projector}. For a single-layer neural network, the integrator evolves the factorized low-rank approximation from time $t_0$ to $t_1 = t_0 + h$ according to
\begin{subequations}\label{eq:dlracont_PSI}
 \begin{align}
\dot{K}(t) =& \, -\nabla \mathcal{\ell}(K(t)V_0^{\top})V_0 &&\textrm{with } K(t_0) = U_0S_0 \, , \label{eq:kstepcont}\\
\dot{S}(t) = & \, U_1^{\top}\nabla \mathcal{\ell}(U_1S(t)V_0^{\top})V_0 && \textrm{with } U_1 S(t_0) = K(t_1) \, , \label{eq:sstepcont}\\
\dot{L}(t) =& \, -\nabla \mathcal{\ell}(U_1L(t)^{\top})^{\top}U_1 &&\textrm{with }L(t_0) = V_0 S(t_1)^\top \, . \label{eq:lstepcont}
\end{align}
\end{subequations}
The factorized solution at $t_1$ is then given by $Y(t_1) = U_1 S_1 V_1^\top$, where $L(t_1) = V_1S_1^\top$ using QR factorization and repeated until the desired end time $t_{\mathrm{end}}$.
For this integrator, a robust error bound is proven by \cite{kieri2016discretized}. A key drawback of this integrator is that \eqref{eq:sstepcont} evolves the solution along the positive gradient direction (or, equivalently, into the reversed time direction of the gradient flow), which can lead to an increase in the loss during the $S$-step.
We will investigate this statement further in Section~\ref{sec:numan}, where we show that the loss cannot be guaranteed to descend because of the $S$-step for the PSI in Lemma \ref{remark:decent_PSI}. 

\subsection{Backward Correction of the PSI}\label{sec_bcpsi}

Addressing the issue of moving backward in time, \cite{bachmayr2021existence} proposed a backward Euler step to update $S$, thereby replacing the reversed time step in the standard PSI. This backward correction of the PSI evolves the factorized low-rank approximation from time $t_0$ to $t_1 = t_0 + h$ according to
\begin{subequations}
\label{eq:corr_PSI_all}
\begin{align} 
\dot{K}(t) =& \, -\nabla \mathcal{\ell}(K(t)V_0^{\top})V_0 &&\textrm{with } K(t_0) = U_0S_0 \, ,\label{eq:PSIM_k}\\
\bar S_1 = & \, U_1^\top U_0S_0&&\textrm{with } U_1R = K(t_1) \, , \label{eq:PSIM_s}\\
\dot{L}(t) =& \, -\nabla \mathcal{\ell}(U_1L(t)^{\top})^{\top}U_1 &&\textrm{with }L(t_0) = V_0 \bar S_1^\top \, .\label{eq:PSIM_l}
\end{align}
\end{subequations}
The factorized solution at $t_1$ is then again given by $Y(t_1) = U_1 S_1 V_1^\top$, where $L(t_1) = V_1S_1^\top$ and repeated until $t_{\mathrm{end}}$. Note that a projection has replaced the evolution equation for $S$; hence, all low-rank factors are evolved forward in time. 

Due to the backward Euler method's consistency, the resulting integrator is expected to retain a robust error bound. For the sake of completeness, we make this statement rigorous in Theorem \ref{theorem:Robust_B_PSI}. We show in Lemma \ref{remark:decent_PSI_B} that for this backward-corrected PSI (bc-PSI), the loss cannot be guaranteed to descend either since 
\begin{align}
\mathcal{\ell}(Y(t_0+h)) \leq \mathcal{\ell}(Y(t_0)) +c_1 \cdot\|(I-U_1U_1^{\top})Y(t_0)\| - h c_2 \, ,
\label{eq:Lemma_3}
\end{align}
with constants $c_1, c_2 > 0$. Note that this result is merely an upper bound and does not guarantee an increase in loss. It, however, provides the necessary understanding to design a novel method that provably fulfills loss descent and converges to a locally optimal point.
\section{The Method: Augmented Backward-Corrected PSI (abc-PSI)}\label{sec:abcPSI}

In this section, we introduce the \textit{augmented backward-corrected PSI (abc-PSI)} which is the main contribution of this paper.
Starting from \eqref{eq:corr_PSI_all}, we keep the K-step \eqref{eq:PSIM_k} and adjust \eqref{eq:PSIM_s} to incorporate a \textit{rank augmentation} step, i.e.,
\begin{align*}
    \bar S_1 =& \, \widehat U_1^\top U_0S_0 &&\textrm{with } \widehat U_1 = \textrm{ortho}([U_0, K(t_1)]),
\end{align*}
where we obtain the orthonormal, augmented basis matrix $ \widehat U_1\in\mathbb{R}^{n\times 2r}$ from the span of the old basis $U_0\in\mathbb{R}^{n\times r}$ and the dynamics of the K-step at final time, $K(t_1)\in\mathbb{R}^{n\times r}$. Here, $\textrm{ortho}$ denotes an orthonormalization process, e.g., computing a QR decomposition and returning the $Q$ factor. 
Projection onto the span of $\widehat U_1$ yields the matrix of augmented coefficients $ \bar S_1 \in\mathbb{R}^{2r\times r}$.

This basis augmentation serves two purposes. First, it is crucial to guarantee loss descent of the abc-PSI, see Theorem \ref{theorem:decent_aB_PSI}, since the problematic term $c_1\cdot\|(I-U_1U_1^{\top})Y(t_0)\|^2$ of \eqref{eq:Lemma_3} vanishes if $\|(I-U_1U_1^{\top})Y(t_0)\|^2 = 0$. Thus, augmenting the basis $U_1$ to also contain the basis vectors of $U_0$, resolves the issue. Second, it allows us to dynamically adjust the rank of the low-rank representation of the weight matrix in combination with a truncation criterion which we introduce at the end of this section. 

The dynamics of the L-step are analogous to the non-augmented bc-PSI of \eqref{eq:corr_PSI_all}. Only the initial condition $L(t_0) = V_0 S_1^\top\in\mathbb{R}^{n\times r}$ is replaced by an augmented initial condition $L(t_0) = V_0 \bar S_1^\top \in\mathbb{R}^{n\times 2r}$. 

In summary, we write the continuous dynamics of the abc-PSI as
\begin{subequations}
\label{eq:abc_aug_bcpsi_all}
\begin{align} 
\dot{K}(t) =& \, -\nabla \mathcal{\ell}(K(t)V_0^{\top})V_0 &&\textrm{with } K(t_0) = U_0S_0 \, ,\label{eq:kstepcont_aug_bc_psi}\\
\bar S_1 =& \, \widehat U_1^\top U_0S_0 &&\textrm{with } \widehat U_1 = \textrm{ortho}([U_0, K(t_1)]) \, , \label{eq:sstepcont_aug_bc_psi}\\
\dot{L}(t) =& \, -\nabla \mathcal{\ell}(\widehat U_1 L(t)^{\top})^{\top} \widehat U_1 &&\textrm{with }L(t_0) = V_0 \bar S_1^\top \, .\label{eq:lstepcont_abc_psi}
\end{align}
\end{subequations}
Due to the augmentation step in \eqref{eq:sstepcont_aug_bc_psi}, the system \eqref{eq:abc_aug_bcpsi_all} doubles in rank at each integration step. To maintain a feasible rank, we dynamically reduce the system's rank by truncating the least important basis vectors of $L(t_1)$ using a truncated singular value decomposition.
To that end, we perform an SVD of $L(t_1)=P\Sigma Q^\top$, with $P\in\mathbb{R}^{n\times 2r}$, $\Sigma\in\mathbb{R}^{2r\times 2r}$, and $Q\in\mathbb{R}^{2r\times 2r}.$
A widely used truncation criterion \cite{schotthofer2022low,ceruti2022rank} to select the rank at the next time step, denoted by $r_1$, is given by
\begin{align*}\label{eq_truncation_threshold}
    \sum_{i=r_1+1}^{2r} \sigma_i^2 < \vartheta,
\end{align*}
where $\sigma_i$ are the singular values of $\Sigma= \textup{diag}(\sigma_1,\dots, \sigma_{2r})$ and $\vartheta$ is the truncation hyper-parameter, which is often formulated as a relative value, i.e., $\vartheta=\tau\Vert {\Sigma}\Vert$. The initial conditions $K_*,$ and $V_*$ for the next iteration of the method is given by 
\begin{subequations}
\begin{align} 
K_* =&\, \widehat U_1 Q^\top_{[1,\dots r_1]} \textup{diag}(\sigma_1,\dots, \sigma_{r_1}) \in\mathbb{R}^{n\times r_1}, \\
V_* =&\, \widehat P_{[1,\dots r_1]} \in\mathbb{R}^{n\times r_1}\,,
\end{align}
\end{subequations}
where $Z_{[1,\dots r_1]}\in\mathbb{R}^{m\times r_1}$ denotes taking the first $r_1$ columns of a matrix $Z^{m\times n}$.
We remark that in total, we require one QR decomposition and one SVD per iteration of the proposed algorithm, whereas methods based on the BUG integrator, e.g. \cite{schotthofer2022low, schotthöfer2024federateddynamicallowranktraining, zangrando2024geometryawaretrainingfactorizedlayers}, or the parallel BUG, e.g. \cite{schotthöfer2024geolorageometricintegrationparameter}, require two QR and one SVD per iteration. This gives the proposed abc-PSI method an advantage in terms of computational cost, since QR and singular value decompositions, though performed for small matrices, are often the main bottleneck of DLRT algorithms.

\subsection{Time integration of the \texorpdfstring{$K$- and $L$-step ODEs}{K- and L-step ODEs}}

The proposed augmented backward-corrected PSI of \eqref{eq:abc_aug_bcpsi_all} contains two differential equation systems in the $K$- and $L$-step. To obtain a practical algorithm, the systems need to be solved with a numerical integrator. Choosing the explicit Euler method as integrator, one obtains the gradient descent method with $\nabla_K \mathcal{\ell}(K_0R)$ indicating the gradient of $\mathcal{\ell}$ with respect to $K$, evaluated at the point $K = K_0$, i.e., we have with $K_1 \approx K(t_1)$ and $L_1 \approx L(t_1)$
\begin{subequations}
\begin{align} 
    {K}_1 &=K_0 -h\nabla_K \mathcal{\ell}(K_0V_0^{\top}), &&\textrm{with }K_0=U_0S_0 \, , \\
    {L}_1 &=L_0 -h\nabla_L \mathcal{\ell} (\widehat U_1 L_0^{\top}), &&\textrm{with }L_0= V_0S_0^\top U_0^\top\widehat U_1\,,
\end{align}
\end{subequations}
where $\widehat U_1 = \text{ortho}([K_0, K(t_1)])$. The updated solution reads $\widehat Y_1 = \widehat U_1 L_1^{\top}$. After the truncation described above, we denote the updated solution as $Y_1 = K_1V_1^{\top}$.
It is straightforward to show that
\begin{align*}
    \nabla_K \mathcal{\ell}(K_0V_0^{\top}) = \nabla \mathcal{\ell}(K_0V_0^{\top})V_0
\qquad \textup{ and}  \qquad 
\nabla_L \mathcal{\ell} (\widehat U_1 L_0^{\top}) = \nabla \mathcal{\ell} (\widehat U_1 L_0^{\top})^{\top}\widehat U_1
\end{align*} 
using the chain rule of differentiation.
We remark that multiple gradient descent steps are compatible with the proposed method. Performing multiple gradient descent steps helps to offset the computational expense of the QR and SVD in the augmentation and truncation steps.
A summary of the method is given in Algorithm \ref{algo:aPSI}. While it is designed for the DLRT of a single-layer network, this simplification is made to streamline the algorithm and can be easily extended to multi-layer networks.

\begin{algorithm}[t!]
\caption{\\Augmented Backward-Corrected Projection Splitting Integration (abc-PSI)}
\label{algo:aPSI}
\begin{algorithmic}
\STATE{\textbf{Input:} Low-rank factorization $Y_0=K_0V_0^\top\in\mathcal{M}_{r_0}$, initial rank $r_0$, and truncation tolerance $\tau>0$.}
\FOR{$k=0,1,\ldots$ and $t_{k+1} = t_k + h$}
    \STATE{\textbf{$K$-step}:}
    \STATE{$K_{k+1}\gets K_{k} -h\nabla_K \ell(K_kV_k^{\top})$}
    \STATE{$\widehat{U}_{k+1},\_ \gets\texttt{QR\_decomposition}([K_k \mid K_{k+1}])$} \hfill \texttt{/* Rank augmentation */}  
    
    \STATE{\textbf{$L$-step}:}
    \STATE{$L_{k}\gets V_k K_k^\top  \widehat{U}_{k+1}$}
    \STATE{$L_{k+1}\gets  L_{k}  -h\nabla_L \ell (\widehat U_{k+1} L_{k})$}
    \STATE{\textbf{Truncation step:}}
    \STATE{$P, \Sigma, Q^\top \gets \texttt{SVD}(L_{k+1})$} 
    \hfill \texttt{/* With $\Sigma = \textup{diag}(\sigma_1, \dots, \sigma_{2r_k})$ */}
    \STATE{Set $r_{k+1} \gets r$ such that $\| [\sigma_{r+1}, \dots, \sigma_{2r_k}] \| \leq \tau \cdot \| [\sigma_1, \dots, \sigma_{2r_k}] \|$}
    \STATE{$K_{k+1} \gets \widehat{U}_{k+1} Q^\top_{[1, \dots, r_{k+1}]} \cdot \textup{diag}(\sigma_1, \dots, \sigma_{r_{k+1}})$}
    \STATE{$V_{k+1} \gets P_{[1, \dots, r_{k+1}]}$}
\ENDFOR
\end{algorithmic}
\end{algorithm}

\section{Loss descent and convergence properties}\label{sec:numan}

In this section, we show that the non-augmented versions, PSI and backward-corrected PSI of \cref{sec_psi} and \cref{sec_bcpsi} respectively, cannot guarantee loss descent. Subsequently, we demonstrate the analytical properties of the abc-PSI using stochastic gradient descent. Although all the following proofs are derived for DLRT of a single-layer network, all results can be directly transferred to multi-layer network training with Proposition 1 of \cite{schotthöfer2024geolorageometricintegrationparameter}.

For the remainder of this paper, $\left\langle \cdot , \cdot \right\rangle$ denotes the scalar product $\langle A, B\rangle = \\ \operatorname{tr}\left(A^T B\right)=\sum_{i, j} a_{i j} b_{i j}$, and $\| \cdot \|$ the Frobenius norm. 
The projection onto the space spanned by $U$, and $V$, are defined by $P_U := U U^\top$, and $P_V := V V^\top$ respectively.

\subsection{Assumptions} \label{ass:all}

For all following proofs, we make Assumptions \ref{ass:delta} - \ref{ass:residual} based on the decomposition of the deterministic gradient $\nabla \mathcal{\ell}(Y)$ into a part $M(Y) \in \mathcal{T}_Y \mathcal{M}_r$ and a residual term $R(Y)$ such that $\nabla \mathcal{\ell}(Y) = M(Y) + R(Y)$.
\vspace{1mm}
\begin{enumerate}[label=(A\arabic*), ref=A\arabic*]
    \item The difference between the initial full-rank and the initial low-rank matrix is bounded by $\delta$, i.e., $\| Y_0 - W_0 \| \leq \delta$. \label{ass:delta}
    \item The stochastic gradient $\nabla \mathcal{\ell}$ is Lipschitz continuous with respect to $\| \cdot \|$ and Lipschitz constant $c_l > 0$. \label{ass:Lipschitz}
    \item The stochastic gradient $\nabla \mathcal{\ell}$ is bounded by a constant $B > 0$. \label{ass:bound}
    \item The residual term $R(Y)$ is bounded by $\epsilon > 0$ for all $Y \in \mathcal{M}_r$. \label{ass:residual}
\end{enumerate}

\subsection{Descent properties of the original PSI}

A descent guarantee of the loss is a central element in proving the convergence of low-rank training methods. While such a property might hold for the original PSI, the descent cannot be proven with standard tools due to the negative $S$-step. It can be shown that the loss decreases in the $K$-step and the $L$-step, while it increases in the $S$-step. To formalize this statement and provide intuition for the dynamics of the PSI, we show the following bound, which is insufficient to prove the convergence of the algorithm.
\begin{lemma}{(Loss evaluation of the PSI)}
\label{remark:decent_PSI}
Let $Y(t)$ be the solution of the PSI evolution equations of \eqref{eq:dlracont_PSI}. Then, the loss is bounded by
\begin{align*} 
\mathcal{\ell}(Y(t_1)) \leq \mathcal{\ell}(Y(t_0)) - \alpha_K^2 h + \alpha_S^2 h  -\alpha_L^2 h  \, 
\end{align*} 
with 
\begin{alignat*}{2}
    \alpha_K=\,&\min\limits_{s\in[t_0,t_1]}\left\| \nabla \mathcal{\ell}(Y_K(s)) V_0 \right\|, \quad\,&&\text{ where } Y_K (t) \coloneqq K(t)V_0^\top\,,\\
    \alpha_S=\,&\max\limits_{s\in[t_0,t_1]}\left\| U_1^\top \nabla \mathcal{\ell}(Y_S(s)) V_0 \right\|, \quad\,&&\text{ where } Y_S (t) \coloneqq U_1 S(t)V_0^\top\,,\\
    \alpha_L=\,&\min\limits_{s\in[t_0,t_1]}\left\| \nabla \mathcal{\ell}(Y_L(s))^\top U_1 \right\|, \quad\,&&\text{ where } Y_L (t) \coloneqq U_1 L(t)^\top\,.
\end{alignat*}

\end{lemma}

\begin{proof}
Following \cite{ceruti2022rank} and \cite{schotthofer2022low}, we investigate the loss decent in all three substeps of \eqref{eq:dlracont_PSI}. Without loss of generality, we prove the bound on the interval $t\in [0,h]$, where $Y(0) =: Y_0 = U_0S_0V_0^{\top}$.
\begin{enumerate}
    \item We first show that the $K$-step \eqref{eq:kstepcont} decreases the loss. Let $Y_K (t) \coloneqq K(t)V_0^\top.$ Then, with \eqref{eq:kstepcont} we have
    \begin{align*}
    \frac{d}{d t} \mathcal{\ell}(Y_K(t)) = & \left\langle\nabla \mathcal{\ell}(Y_K(t)), \dot{Y}_K(t)\right\rangle \\
    = & \left\langle\nabla \mathcal{\ell}(Y_K(t)), \dot{K}(t) V_0^{\top}\right\rangle \\
    \overset{\eqref{eq:kstepcont}}{\hfill =} & \left\langle\nabla \mathcal{\ell}(Y_K(t)),  -\nabla\mathcal{\ell}(K(t)V_0^{\top})V_0 V_0^{\top}\right\rangle\\
    = & - \left\langle\nabla \mathcal{\ell}(Y_K(t))V_0,  \nabla\mathcal{\ell}(K(t)V_0^{\top})V_0\right\rangle \\
    = & -\left\| \nabla \mathcal{\ell}(Y_K(t))V_0 \right\|^2 \, .
    \end{align*}
    With $\alpha_K=\min_{0 \leq \tau \leq 1}\left\| \nabla \mathcal{\ell}(Y_K(\tau h)) V_0 \right\|$, we have $\frac{d}{d t} \mathcal{\ell}(Y_K(t)) \leq - \alpha_K^2$. Taking the integral from $t_0 = 0$ to $t_1 = h$ yields, with $\int_{0}^{h} \frac{d}{d t} \mathcal{\ell}(Y_K(t)) dt = \mathcal{\ell}(Y_K(t_1)) - \mathcal{\ell}(Y_0)$,
    \begin{align*}
        \mathcal{\ell}\left(Y_K(t_1)\right) \leq \mathcal{\ell}\left(Y_0\right) - \int_{0}^{h} \alpha_K^2 dt = \mathcal{\ell}\left(Y_0\right) - \alpha_K^2 h \, .
    \end{align*}
%
    \item We then show that the loss increases in the $S$-step \eqref{eq:sstepcont}. Let $Y_S (t) \coloneqq U_1 S(t)V_0^\top$. Then, with \eqref{eq:sstepcont} we know that
    \begin{align*}
    \frac{d}{d t} \mathcal{\ell}(Y_S(t)) = & \langle\nabla \mathcal{\ell}(Y_S(t)), \dot{Y}(t)\rangle \\
    = & \left\langle\nabla \mathcal{\ell}(Y_S(t)), U_1 \dot{S}(t) V_0^{\top}\right\rangle \\
    \overset{\eqref{eq:sstepcont}}{\hfill =} & \left\langle U_1^{\top} \nabla \mathcal{\ell}(Y_S(t)) V_0, \dot{S}(t)\right\rangle \\
    = & \left\langle U_1^{\top} \nabla \mathcal{\ell}(Y_S(t)) V_0,U_1^{\top} \nabla \mathcal{\ell}(Y_S(t)) V_0\right\rangle\\
    = & \left\|U_1^{\top} \nabla \mathcal{\ell}(Y_S(t)) V_0\right\|^2 \, .
    \end{align*}
    With $\alpha_S=\max_{0 \leq \tau \leq 1}\left\| U_1^\top \nabla \mathcal{\ell}(Y_S(\tau h)) V_0 \right\|$, we have  
    $\frac{d}{d t} \mathcal{\ell}(Y_S(t)) \leq \alpha_S^2$. Taking the integral from $t_0 = 0$ to $t_1 = h$ yields, with $\int_{0}^{h} \frac{d}{d t} \mathcal{\ell}(Y_S(t)) dt = \mathcal{\ell}(Y_S(t_1)) - \mathcal{\ell}(Y_K(t_1))$,
    \begin{align*}
       \mathcal{\ell}\left(Y_S(t_1)\right) \leq \mathcal{\ell}\left(Y_0\right) - h \alpha_K^2 + h \alpha_S^2 \, .  
    \end{align*}
%
    \item We show that the $L$-step \eqref{eq:lstepcont} decreases the loss analogously to the $K$-step. Let $Y_L (t) \coloneqq U_1 L(t)^\top$. As for the $K$-step we have
    \begin{align*}
    \frac{d}{d t} \mathcal{\ell}(Y_L(t)) = -\left\| \nabla \mathcal{\ell}(Y_L(t))^\top U_1 \right\|^2 \, .
    \end{align*}
    With $\alpha_L=\min_{0 \leq \tau \leq 1}\left\| \nabla \mathcal{\ell}(Y_L(\tau h))^\top U_1 \right\|$, we have $\frac{d}{d t} \mathcal{\ell}(Y_L(t)) \leq - \alpha_L^2$. 
    
    Hence,
    \begin{align*}
        \mathcal{\ell}\left(Y(t_1)\right) = \mathcal{\ell}\left(Y_L(t_1)\right) \leq \mathcal{\ell}\left(Y_0\right) - h \alpha_K^2 + h \alpha_S^2 - h \alpha_L^2 \, .
    \end{align*}
\end{enumerate}
\end{proof}
\begin{remark}
\label{remark:counter_example_PSIorig}
The derivation shows that if 
\begin{align*}
    \int_0^h \left\| \nabla \mathcal{\ell}(Y_K(t)) V_0 \right\|^2\, dt + \int_0^h \left\| \nabla \mathcal{\ell}(Y_L(t))^\top U_1 \right\|^2\,dt \leq \int_0^h \left\|U_1^{\top} \nabla \mathcal{\ell}(Y_S(t)) V_0\right\|^2\,dt \, ,
\end{align*}
then the loss increases over one time step.
\end{remark}

\subsection{Robust Error Bound of the Backward-Corrected PSI}

The original PSI has already been shown to have a robust error bound even in the presence of small singular values \cite{kieri2016discretized}. Therefore, a similar robust error bound is expected to hold for its version in which one of the substeps changed to an implicit time discretization. For completeness, we present a rigorous proof of the robustness of the backward-corrected PSI of \cref{sec_bcpsi}. To analyze the robust error bound of the backward-corrected PSI, we first show that the deviation between the PSI and bc-PSI is sufficiently small for all steps K, S, and L and then conclude the robustness following the proof of the robust error bound of the original PSI \cite[Theorem~2.1]{kieri2016discretized}.

\begin{theorem}{(Robust error bound of the bc-PSI)}
\label{theorem:Robust_B_PSI}
Let us denote the weights at time $t_n = t_0 + nh$ following the original gradient flow \eqref{eq:gradflow} as $W(t_n)$ and the weights of the backward-corrected PSI following the evolution equations \eqref{eq:corr_PSI_all} as $\bar Y_n$.
Under Assumptions~\ref{ass:all}, the global error is bounded by
\begin{align*}
\| W(t_n) - \bar Y_n \| \leq c_1 h + c_2 \varepsilon + c_3 \delta\, ,
\end{align*}
where $c_{1,2,3}$ are independent of singular values of the numerical and exact solution.
\end{theorem}
\begin{proof}
In the following, let all variables overset by $\sim$ describe variables taken from the original PSI, while all variables overset by $-$ describe variables taken from the bc-PSI. Moreover, let us denote an arbitrarily chosen time $t_{n-1}$ as $t_0$ and $t_n$ as $t_1$. We start by bounding the distance of the results from the PSI and the bc-PSI in all three substeps where we assume that both integrators start with the same initial condition $Y_0 = U_0S_0V_0^{\top}$. That is, we start by investigating the local error in the following four steps.
\begin{enumerate}
    \item The $K$-step of both integrators is the same. Thus, $\widetilde{K}(t) = \bar{K}(t) =: K(t)$ for $t\in[t_0, t_1]$.
    \item \noindent We note that for the original PSI, multiplying $K_1 = K(t_1)$  
    with $V_0^{\top}$ and $\bar S(t_1)$ 
    with $U_1$ and $V_0^{\top}$ yields
\begin{align}
K_1V_0^{\top} =\,& Y_0-\int_{t_0}^{t_1}\nabla\mathcal{\ell}(K(t)V_0^{\top})V_0V_0^{\top}\,dt \label{eq:stability_s1}
\end{align}
for the $K$-step, and 
\begin{align}
U_1\widetilde S_1V_0^{\top} =\,& K_1V_0^{\top} + \int_{t_0}^{t_1} U_1U_1^{\top}\nabla\mathcal{\ell}(U_1\widetilde S(t)V_0^{\top})V_0V_0^{\top}\,dt\, \label{eq:stability_s2}
\end{align}
for the $S$-step. Next, we plug \eqref{eq:stability_s1} into \eqref{eq:stability_s2}, which yields
\begin{align}
U_1\widetilde S_1V_0^{\top} =Y_0\,&-\int_{t_0}^{t_1}\nabla\mathcal{\ell}(K(t)V_0^{\top})V_0V_0^{\top}\,dt\nonumber\\
\,&+ \int_{t_0}^{t_1} U_1U_1^{\top}\nabla\mathcal{\ell}(U_1\widetilde S(t)V_0^{\top})V_0V_0^{\top}\,dt\, .\label{eq:S2}
\end{align}
\noindent We add and subtract $\int_{t_0}^{t_1}U_1U_1^{\top}\nabla\mathcal{\ell}(K(t)V_0^{\top})V_0V_0^{\top}\,dt$ as well as define 
\begin{align*}
\Delta := \int_{t_0}^{t_1}\nabla\mathcal{\ell}(K(t)V_0^{\top})\,dt-\int_{t_0}^{t_1} \nabla\mathcal{\ell}(U_1\widetilde S(t)V_0^{\top})\,dt\,.
\end{align*}
Then, \eqref{eq:S2} becomes
\begin{align}\label{eq:S3mid}
U_1\widetilde S_1V_0^{\top} = K_0V_0^{\top}-\,&(I-U_1U_1^{\top})\int_{t_0}^{t_1}\nabla\mathcal{\ell}(K(t)V_0^{\top})\,dt V_0V_0^{\top}\nonumber\\
+\,& U_1U_1^{\top}\Delta V_0V_0^{\top}\, .
\end{align}
We note that
\begin{align*}
    \Vert \Delta\Vert \leq\,& c_l\int_{t_0}^{t_1}\Vert K(t)V_0^{\top}- U_1\widetilde S(t)V_0^{\top}\Vert\,dt\\
    \leq\,& c_l\int_{t_0}^{t_1}\Vert K(t_1)V_0^{\top}- U_1\widetilde S(t_0)V_0^{\top}\Vert\,dt\\
    \,& + c_l\int_{t_0}^{t_1}\int_{t_0}^{t}\Vert \dot K(s)V_0^{\top}+ U_1\dot{\widetilde S}(s)V_0^{\top}\Vert\,dsdt
    \leq 2c_lB h^2\,.
\end{align*}
Multiplication of \eqref{eq:S3mid} with $U_1^{\top}$ and $V_{0}$ yields
\begin{align*}
\widetilde S_1 = U_1^{\top}K_0 + U_1^{\top}\Delta V_0\, .
\end{align*}
Using Assumption \ref{ass:Lipschitz} and recalling that $ \bar S_1 = U_1 ^\top K_0$, we have 
\begin{align*}
\| \widetilde S_1 - \bar S_1 \| \leq  \| U_1^{\top}\Delta V_0\| \leq 2c_l B h^2\, .
\end{align*}
    \item With Assumption \ref{ass:Lipschitz} and the orthogonality of the columns in $V_0, U_1$, the distance of the results from the PSI and the bc-PSI after the $L$-step is bounded by
\begin{align*}
    \| \widetilde{L}_1 - \bar{L}_1 \|
    \leq\,& \| \, \widetilde{L}_0 - \bar{L}_0 \| +  \int_{t_0}^{t_1} \|( \nabla \mathcal{\ell}(U_1  \widetilde{L}(t)^{\top})^{\top}  - \nabla \mathcal{\ell}(U_1  \bar{L}(t)^{\top})^{\top}) U_1\|\,dt   \\
    \leq\,& \, \| V_0 (\widetilde{S}_1 - \bar{S_1})^{\top} \| + h c_l \| \widetilde{L}_0 - \bar{L}_0\| 
    +c_l\int_{t_0}^{t_1}\int_{t_0}^{t}\Vert \dot{\widetilde{L}}(s)- \dot{\bar{L}}(s)\Vert\,dsdt
    \\
    \leq\,& \, 2c_l B h^2 + 2c_l^2 B h^3 +c_lBh^2 \, .
\end{align*}
\item Hence, we have that $\Vert \widetilde Y_1 - \bar Y_1 \Vert \leq 2c_l B h^2 + 2c_l^2 B h^3 +c_lBh^2$. Then, according to \cite[Theorem~2.1]{kieri2016discretized}, the local error is bounded by
\begin{align}
    \Vert W(t_1) - \bar Y_1 \Vert \leq \Vert W(t_1) - \widetilde Y_1 \Vert + \Vert \widetilde Y_1 - \bar Y_1 \Vert \leq c_1 h^2 + c_2 h \varepsilon\,.
\end{align}
\end{enumerate}
Concluding the proof, the result on the global error $\| W(t_n) - \bar Y_n \|$ follows from the Lady Windermere’s fan argument \cite[II.3]{norsett1987solving} with error propagation via the exact flow; cf. \cite{ceruti2022rank, ceruti2022unconventional, kieri2016discretized, kieri2019projection, ceruti2024robust}.
\end{proof}

\subsection{Descent properties of the backward-corrected PSI}
\begin{lemma}{(Loss evaluation of the bc-PSI)}
\label{remark:decent_PSI_B}
Under Assumption \ref{ass:Lipschitz} and \ref{ass:bound}, let $Y(t)$ be the solution of the backward-corrected PSI evolution equations of \eqref{eq:corr_PSI_all}. Then, the loss is bounded by
\begin{align*} 
\mathcal{\ell}(Y(t_1)) \leq \mathcal{\ell}(Y_0) +B \|(I-U_1U_1^{\top})Y_0\|+\frac{c_l}{2}\|(I-U_1U_1^{\top})Y_0\|^2 - h \alpha_L^2 \, 
\end{align*} 
with $\alpha_L=\min\limits_{s\in[t_0,t_1]}\left\| \nabla \mathcal{\ell}(Y_L(s))^\top U_1 \right\|$.
\end{lemma}

\begin{proof}
    As before, we investigate the time interval $[0,h]$. We start with the $L$-step, which analogously to the proof of Lemma~\ref{remark:decent_PSI} gives with $Y_L(0) = U_1U_1^{\top}Y_0$
    \begin{align*}
            \mathcal{\ell}\left(Y(t_1)\right) = \mathcal{\ell}\left(Y_L(t_1)\right) \leq \mathcal{\ell}\left(U_1U_1^{\top}Y_0\right) - h \alpha_L^2 \, .
    \end{align*}
    Using Assumption \ref{ass:Lipschitz} and Lemma 5.2. of \cite{arsen_stachastic_gradient_descent} yields for general $Z_1, Z_2\in\mathbb{R}^{m\times n}$
    \begin{align*}
    \mathcal{\ell}(Z_1) \leq \mathcal{\ell}(Z_2)-\langle \nabla \mathcal{\ell}(Z_2), Z_1-Z_2\rangle+\frac{c_l}{2}\|Z_1-Z_2\|^2\,.
    \end{align*}
    Then, using the above inequality with $Z_1 = U_1U_1^{\top}Y_0$ and $Z_2 = Y_0$ as well as the Cauchy-Schwartz inequality and boundedness of $\nabla \mathcal{\ell}$ yields
    \begin{align*}
            \mathcal{\ell}\left(Y(t_1)\right)
            \leq\,& \mathcal{\ell}(Y_0)+\langle \nabla \mathcal{\ell}(Y_0), (I-U_1U_1^{\top})Y_0\rangle+\frac{c_l}{2}\|(I-U_1U_1^{\top})Y_0\|^2 - h \alpha_L^2 \, \\
            \leq\,& \mathcal{\ell}(Y_0)+B \|(I-U_1U_1^{\top})Y_0\|+\frac{c_l}{2}\|(I-U_1U_1^{\top})Y_0\|^2 - h \alpha_L^2 \, .
    \end{align*}
\end{proof}

While this result does not immediately show a decrease or increase in the loss, it directly shows how to adapt the method to guarantee descent. The term that can potentially increase the loss (or at least render our analytic result impractical) is $\|(I-U_1U_1^{\top})Y_0\|^2$.

\subsection{Robustness of the abc-PSI}

The previous derivations have shown that while the bc-PSI has a robust error bound, showing loss descent remains difficult. Loss-descent is, however, a key ingredient in proving convergence to a local low-rank optimum. In this section, we show that the abc-PSI does not suffer from this problem. Throughout the following proofs we denote the solution of the abc-PSI before truncation as $\widehat Y_n = \widehat U_n L(t_n)^{\top}$ and after truncation as $Y_n = U_n S_n V_n^{\top}$ where $\Vert \widehat Y_n - Y_n \Vert \leq \vartheta$. As in the previous sections, we investigate the time interval $[t_0, t_1]$ for ease of presentation. Moreover, we use several properties of the projector $P_{\widehat U_{k+1}} = \widehat U_{k+1} \widehat U_{k+1}^\top$ which we state in the following.
\begin{remark}\label{remark:projectionIsYero}
Using the augmented basis $\widehat{U}_{k+1}$ yields
\begin{subequations}
\begin{align}
    P_{\widehat U_{k+1}} \widehat U_{k+1} = \, &\widehat U_{k+1} \widehat U_{k+1} ^\top \widehat U_{k+1} = \widehat U_{k+1} \, ,\\
    P_{\widehat U_{k+1}} U_{k} = \, & \widehat U_{k+1} \widehat U_{k+1} ^\top U_{k} = U_{k} \, . 
\end{align}
\end{subequations}
Thus, it holds that 
\begin{align} 
(I - P_{\widehat U_{k+1}}) \widehat U_{k+1} = 0 \, . \label{eq:55}
\end{align}
\end{remark}
Special applications of \eqref{eq:55} are $(P_{\widehat U_{k+1}} - I) K = 0$ and $(P_{\widehat U_{k+1}} - I) Y = 0$, for $K = U_k S_k$ and $Y = U_k S_k V_k^\top$. 
Note that because $\widehat V_{k+1}$ does not necessarily contain the basis vectors $V_k$, these equations do not hold for $P_{\widehat V_{k+1}}$. I.e.,
\begin{align*}
   0 = \left(I-P_{\widehat V_{k+1}}\right) L\left(t_{k+1}\right) \neq \left(I-P_{\widehat V_{k+1}}\right) L\left(t_k\right) \, . \label{eq:L_IP_zero}
\end{align*}
Note that the following results, namely the robust error bound, loss descent, and convergence of the augmented backward-corrected PSI, are shown for the discrete Algorithm \ref{algo:aPSI}. These properties are not satisfied by the PSI and bc-PSI.
\begin{theorem}{(Robust error bound of the abc-PSI)}
\label{theorem:Robust_aB_PSI}
Let $Y(t_n)$ denote the solution of \cref{eq:abc_aug_bcpsi_all} when using the stochastic gradient, and $W(t_n)$ denote the solution of the full-rank gradient flow \eqref{eq:gradflow} at time $t_n$.
Under Assumptions \ref{ass:all}, the global error is bounded by
\begin{align*}
\| Y(t_n) - W(t_n) \| \leq \epsilon + c_1 h + c_2 \delta + \frac{\vartheta}{h} \, ,
\end{align*}
where $c_{1,2}$ are independent of singular values in the exact and numerical solutions.
\end{theorem}
\begin{proof}
To bound the distance between the low-rank solution $Y(t)$ and the full-rank solution $W(t)$ after one time step from $t_0$ to $t_1 = t_0 + h$ when starting at the same initial condition, i.e., $W(t_0) = Y(t_0)$, we get
\begin{align}
     \| \widehat Y_1 - W(t_1) \| =\,& \| \widehat U_1 L (t_1)^\top - W(t_1) \| \notag \\
     =\, &\| \widehat U_1 L (t_0)^\top + \int_{t_0}^{t_1} \widehat U_1 \dot L (t)^\top dt - W(t_0) - \int_{t_0}^{t_1} \dot W(t) dt \| \notag\\
     \leq \,& \int_{t_0}^{t_1} \| \widehat U_1 \dot L (t)^\top - \dot W(t) \| dt  \, . \label{eq:T2_block1}
\end{align}
Note that we used $\widehat U_1 L_0^\top = W(t_0)$. Plugging in $\dot L (t)^\top$ from \eqref{eq:lstepcont_abc_psi} into \eqref{eq:T2_block1} yields
\begin{align*}     
     \| \widehat Y_1 - W(t_1) \| \leq & \int_{t_0}^{t_1} \| \widehat U_1 \widehat U_1 ^\top \stochgrad(\widehat U_1 L(t)^\top) - \stochgrad(W(t))\| dt \, .
\end{align*}
With zero completion, the orthogonality of the columns of $\widehat U_1$, and Assumption \ref{ass:Lipschitz}, we get
\begin{align}
    \| \widehat Y_1 - W(t_1) \| \leq & \, \int_{t_0}^{t_1} \| \widehat U_1 \widehat U_1 ^\top \stochgrad(\widehat U_1 L(t)^\top) - \widehat U_1 \widehat U_1 ^\top \stochgrad(Y_0) \| dt \notag \\
    & \, + \int_{t_0}^{t_1} \|  \widehat U_1 \widehat U_1 ^\top \stochgrad(Y_0) -\stochgrad(W(t))\| dt  \notag\\
    \leq & \, \int_{t_0}^{t_1} \| \stochgrad(\widehat U_1 L(t)^\top) - \stochgrad(Y_0) \| dt \notag\\
    & \, + \int_{t_0}^{t_1} \|  \widehat U_1 \widehat U_1 ^\top \stochgrad(Y_0) -\stochgrad(W(t))\| dt  \notag\\
    \leq & \,  c_l \int_{t_0}^{t_1} \|  \widehat U_1 L(t)^\top - Y_0 \| dt + \int_{t_0}^{t_1} \| \widehat U_1 \widehat U_1 ^\top \stochgrad(Y_0) -\stochgrad(W(t))\| dt \, . \label{eq:T2_block2}
\end{align}
Using $L(t)^\top = L_0^\top - \int_{s_0}^s \widehat U_1^\top \stochgrad(\widehat U_1 L(s)) ds$ and $\widehat U_1 L_0^\top = Y_0$, yields
\begin{align*} 
     \int_{t_0}^{t_1} \|  \widehat U_1 L(t)^\top - Y_0 \| dt
    =& \, \int_{t_0}^{t_1} \| \widehat U_1 L_0^\top - \widehat U_1 \int_{s_0}^s \widehat U_1^\top \stochgrad(\widehat U_1 L(s)) ds - Y_0 \| dt \\
    \leq& \, \int_{t_0}^{t_1}\int_{s_0}^s \| \stochgrad(\widehat U_1 L(s))  \|\,ds \, dt \, .
\end{align*}
Then, with Assumption \ref{ass:bound}, stating that $\| \stochgrad \| \leq B$
\begin{align} 
    \, c_l \int_{t_0}^{t_1} \int_{s_0}^s \| \stochgrad(\widehat U_1 L(s)) \| ds \, dt
    \leq \, c_l \int_{t_0}^{t_1} \int_{t_0}^s B \, ds \, dt 
    \leq \, c_l B h^2 \, . \label{eq:T2_cBh2}
\end{align} 
Moreover, the second term in \eqref{eq:T2_block2} can be bounded by 
\begin{align*} 
\| \widehat U_1 \widehat U_1 ^\top \stochgrad(Y_0) -\stochgrad(W(t))\| \leq & \| \widehat U_1 \widehat U_1 ^\top \stochgrad(Y_0) - \stochgrad(Y_0) \| + \| \stochgrad(Y_0) - \stochgrad(W(t)) \|  \\
\leq & \| \widehat U_1 \widehat U_1 ^\top \stochgrad(Y_0) - \stochgrad(Y_0) \| + c_l \| Y_0 - W(t) \|  \, .
\end{align*} 
With Taylor-Expansion we have $\| Y_0 - W(t) \| \leq B h$. Then, using this and \eqref{eq:T2_cBh2}, the inequality \eqref{eq:T2_block2} becomes 
\begin{align*} 
    \| \widehat Y_1 - W(t_1) \| \leq \int_{t_0}^{t_1} \| \widehat U_1 \widehat U_1 ^\top \stochgrad(Y_0) - \stochgrad(Y_0)\| dt + 2 c_l B h^2  \, .
\end{align*}
Using $\stochgrad(Y) = M(Y) + R(Y)$ yields
\begin{align*}
    \| \widehat Y_1 - W(t_1) \| \leq\, & \, h \| \widehat U_1 \widehat U_1 ^\top \stochgrad(Y_0) - \stochgrad(Y_0)\| + 2 c_l B h^2  \\
     \leq & \, h \| (\widehat U_1 \widehat U_1 ^\top - I ) M(Y_0) \| +  h\|  (\widehat U_1 \widehat U_1 ^\top - I ) R(Y_0) \| + 2 c_l B h^2  \, .
\end{align*}
With $M(Y_0) = P(Y_0)\stochgrad(Y_0) = U_0 U_0 ^\top \stochgrad(Y_0) - U_0 U_0 ^\top \stochgrad(Y_0) V_0 V_0 ^\top + \stochgrad(Y_0) V_0 V_0 ^\top$ and $(\widehat U_1 \widehat U_1 ^\top - I )U_0 = 0$, we get
\begin{align*}
     (\widehat U_1 \widehat U_1 ^\top - I )M(Y_0) = (\widehat U_1 \widehat U_1 ^\top - I )\stochgrad(Y_0) V_0 V_0 ^\top \, .
\end{align*}
Using that $(\widehat U_1 \widehat U_1 ^\top - I ) K(t_1) = 0$ and $(\widehat U_1 \widehat U_1 ^\top - I ) Y_0 = 0$, since $K(t_1)$ and $Y_0$ are spanned by $\widehat U_1$ this yields 
\begin{align}
     \| \widehat Y_1 - W(t_1) \| \leq\,& \, h \| (\widehat U_1 \widehat U_1 ^\top - I ) \stochgrad(Y_0) V_0 V_0^\top \| +  h \epsilon + 2 c_l  B h^2  \notag\\
     = & \, h \left\| (\widehat U_1 \widehat U_1 ^\top - I ) \left(\stochgrad(Y_0) V_0 V_0^\top + \frac{1}{h}(K(t_1) V_0^\top - Y_0)\right)\right\|  \nonumber \\
     & \, +  h \epsilon + 2 c_l B h^2   \, , \label{eq:T2_block3}
\end{align}
 {Lastly, we bound the norm on the right-hand side. Let us note that
\begin{align*}
    K(t_1) V_0^\top - Y_0 =\,& -\int_{t_0}^{t_1} \stochgrad(K(t) V_0^\top) V_0 V_0 ^\top \,dt\\
    =\,& -h \stochgrad(Y_0^\top) V_0 V_0 ^\top -\int_{t_0}^{t_1} (\stochgrad(K(t) V_0^\top) - \stochgrad(Y_0)) V_0 V_0 ^\top \,dt\,.
\end{align*}
Together with the orthonormality of $(\widehat U_1 \widehat U_1 ^\top - I )$, the norm in \eqref{eq:T2_block3} is bounded by
\begin{align*}
    \left\| \stochgrad(Y_0) V_0 V_0^\top + \frac{1}{h}(K(t_1) V_0^\top - Y_0)\right\| \leq\,& \frac{1}{h}\int_{t_0}^{t_1} \Vert \stochgrad(K(t) V_0^\top) - \stochgrad(Y_0) \Vert \,dt \\
    \leq\,& \frac{c_l}{h} \int_{t_0}^{t_1} \Vert K(t) V_0^\top - Y_0 \Vert \,dt\\
    \leq\,& \frac{c_l}{h} \int_{t_0}^{t_1}\int_{t_0}^{s} \Vert \dot K(s) V_0^\top\Vert \,ds dt\,.
\end{align*}
Hence, since $\Vert \dot K(s) V_0^\top\Vert \leq B$, the above term is bounded by $c_l Bh$. Plugging this into \eqref{eq:T2_block3} gives}
\begin{align}
     \| \widehat Y_1 - W(t_1) \| \leq c_l Bh^2  +  h \epsilon + 2 c_l B h^2   \, .
\end{align}
Hence, after truncation, the local error is bounded by
\begin{align*}
\| Y_1 - W(t_1) \| \leq \| \widehat Y_1 - W(t_1) \| + \| \widehat Y_1 - Y_1 \| \leq h \epsilon + 3 c_l B h^2 + \frac{\vartheta}{h} \, .
\end{align*}
Concluding the proof, the result on the global error $ \| \widehat Y_1 - W(t_1) \|$ follows from applying the standard Lady Windermere’s fan argument \cite[II.3]{norsett1987solving} with error propagation via the exact flow; cf. \cite{ceruti2022rank, ceruti2022unconventional, kieri2016discretized, kieri2019projection, ceruti2024robust}.
\end{proof}

%
% ...
%
%
\subsection{Discrete Case: Upper Bound of the Loss Function using SGD}

In the following, we restate Lemma 5.2. of \cite{arsen_stachastic_gradient_descent}. This lemma holds for the stochastic as well as the deterministic gradient.
\begin{lemma}\label{lemma:loss}
Under Assumption \ref{ass:Lipschitz}, for any $Z_1, Z_2\in\mathbb{R}^{m\times n}$ it holds that
\begin{align*}
\ell(Z_1) \leq \ell(Z_2)-\langle \nabla \ell(Z_2), Z_1-Z_2\rangle+\frac{c_l}{2}\|Z_1-Z_2\|^2 \, .
\end{align*}
\end{lemma}
The proof can be found in appendix \ref{app:lemma4}. 
With this, we show loss descent for sufficiently small learning rates $h \leq \frac{2}{c_l}$.
%
%
\begin{theorem}{(Loss descent of the abc-PSI)}\label{theorem:decent_aB_PSI}
Under Assumption \ref{ass:Lipschitz}, the loss of the low-rank solution $Y$ calculated with the stochastic augmented backward-corrected PSI as in \eqref{eq:corr_PSI_all} and Algorithm \ref{algo:aPSI} using the stochastic gradient is 
\begin{align}\label{eq:lossdescent_abc_PSI}
    \ell(\widehat Y_1) \leq \ell(Y_0) - \left(1 - \frac{h c_l}{2}\right) h\| P_{\widehat U_1} \nabla \ell(Y_0)\|^2 \, .
\end{align}
\end{theorem}
\begin{proof}
We have $\widehat Y_1 = \widehat U_1 L_1^\top$, where
\begin{align*}
L_1^\top
= L_0^\top - h \widehat U_1^\top \nabla \ell(\widehat U_1 L_0^\top ) \, .
\end{align*}
Multiplying both sides with $\widehat U_1$, yields 
\begin{align}
\widehat U_1 L_1^\top = & \widehat U_1 L_0^\top - h \widehat U_1 \widehat U_1^\top \nabla \ell(\widehat U_1 L_0^\top ) \, . \label{eq:T3_begin}
\end{align}
Then, using $\widehat Y_1 = \widehat U_1 L_1^\top, Y_0 = \widehat U_1 L_0^\top$, and $P_{\widehat U_1} = \widehat U_1 \widehat U_1^\top$, \eqref{eq:T3_begin} becomes
\begin{align*}
\widehat Y_1 = & Y_0 - h P_{\widehat U_1}\nabla \ell(Y_0) \, .
\end{align*}
With this and Lemma \ref{lemma:loss}, using $Z_1 = \widehat Y_1$ and $Z_2 = Y_0$,
\begin{align*}
    \ell(\widehat Y_1)-\ell(Y_0) 
    = & \, \ell(Y_0 - h P_{\widehat U_1} \nabla \ell(Y_0) )-\ell(Y_0)\\
    \leq & \, \ell(Y_0)+\langle \nabla \ell(Y_0), Y_0 - h P_{\widehat U_1} \nabla \ell(Y_0) -Y_0\rangle\\
    &+\frac{c_l}{2}\| Y_0 - h P_{\widehat U_1} \nabla \ell(Y_0) -Y_0\|^2 -\ell(Y_0) \\
    = & - h \langle \nabla \ell(Y_0), P_{\widehat U_1} \nabla \ell(Y_0)\rangle+\frac{h^2 c_l}{2}\| P_{\widehat U_1} \nabla \ell(Y_0)\|^2 \\
    = & - h \| \widehat P_{\widehat U_1} \nabla \ell(Y_0)\|^2 + \frac{h^2 c_l}{2}\| P_{\widehat U_1} \nabla \ell(Y_0)\|^2 \, .  
\end{align*}
\end{proof}

\subsection{Convergence of the abc-PSI}\label{app_convergence}

Given the previous discussion, we can now conclude that Algorithm~\ref{algo:aPSI} converges to weights that satisfy the local optimality criterion for optimization on manifolds, see, e.g., \cite[Theorem~3.4]{sato2021riemannian}. In the following, we assume that the learning rate can vary with respect to the iteration index, denoted by $h_t$. Under the Robbins-Monro conditions, we proceed to prove convergence.
\begin{theorem}{(Convergence of the abc-PSI)}
    Under Assumption \ref{ass:Lipschitz}, \ref{ass:bound}, let $\ell\geq 0$ and $Y_{t}$ for $t\in\mathbb{N}$ be the solutions obtained from Algorithm~\ref{algo:aPSI}. Let the learning rate sequence $(h_t)_{t\in\mathbb{N}}$ satisfy the Robbins-Monro conditions
    \[
    \textstyle{\sum_{t=1}^{\infty} h_t =+\infty \, , \qquad \sum_{t=1}^{\infty} h_t^2 < +\infty \, ,}
    \]
and let $\sum_{t=1}^{T}\mathbb{E}[\|Y_{t} - \widehat Y_{t} \|] \leq D < \infty$, i.e.,  for sufficiently large $t$, the rank stabilizes. 
Then, algorithm \ref{algo:aPSI} using the stochastic gradient $\nabla \ell$ converges to locally optimal weights, i.e., 
       \begin{align*}
        \liminf_{T\rightarrow\infty} \mathbb{E}[\Vert P(Y_{T})\nabla\ell(Y_{T}) \Vert^2] = 0\,,
        \end{align*}
    with expected values taken over all $\xi_t$.
\end{theorem}

\begin{proof}
    The proof adapts the proofs of \cite{arsen_stachastic_gradient_descent} and \cite{schotthöfer2024geolorageometricintegrationparameter} for the proposed integrator. For a general iteration step $t$, we have with \eqref{eq:lossdescent_abc_PSI}
    \begin{align*}
        \ell (\widehat Y_t) \leq \ell (Y_{t-1}) - \left(1 - \frac{h c_l}{2}\right) h  \| \widehat U_t\widehat U_t^\top \nabla\ell(Y_{t-1}) \|^2\,.
    \end{align*}
     Taking the expected value over $\xi_{1},\dots,\xi_{T}$ and denoting the corresponding expected value as $\mathbb{E}[\cdot]$ yields
    \begin{align*}
        \mathbb{E}[\ell(Y_t)] - \mathbb{E}[\ell(Y_{t-1})] \leq\,&  -{h_t}\mathbb{E}[\Vert \widehat U_t\widehat U_t^\top \nabla\ell(Y_{t-1}) \Vert^2] + \frac{c_l{h_t}^2}{2}\mathbb{E}[\Vert \widehat U_t\widehat U_t^\top \nabla\ell(Y_{t-1}) \Vert^2]\\
        \,&+ c_l \mathbb{E}[\|Y_t - \widehat Y_{t} \|] \, \\
        = \,&-h_t\left(1-\frac{c_l {h_t}}{2}\right)\mathbb{E}[\Vert \widehat U_t\widehat U_t^\top \nabla\ell(Y_{t-1}) \Vert^2]  + c_l \mathbb{E}[\|Y_t - \widehat Y_{t} \|] \,.
    \end{align*}
    Summing over $t=1, \dots,T$ and using the telescoping sum on the left-hand side then yields
    \begin{align*}
        - \ell(Y_{{0}}) \leq \mathbb{E}[\ell(Y_{T})] - \ell(Y_{{0}}) \leq \,&-\sum_{t=1}^{T}h_t\left(1-\frac{c_l {h_t}}{2}\right)\mathbb{E}[\Vert \widehat U_t\widehat U_t^\top \nabla\ell(Y_{t-1}) \Vert^2]\\
        \,&+ c_l \sum_{t=1}^{T}\mathbb{E}[\|Y_t - \widehat Y_{t} \|] \,.
    \end{align*}
    With $\sum_{t=1}^{T}\mathbb{E}[\|Y_t - \widehat Y_t \|]\leq D$ we can rearrange the above inequality as
        \begin{align}\label{eq:conv1}
        \sum_{t=1}^{T}h_t \left(1-\frac{c_l {h_t}}{2}\right)\mathbb{E}[\Vert \widehat U_t\widehat U_t^\top \nabla\ell(Y_{t-1}) \Vert^2] 
        \leq & \, \ell(Y_{{0}}) + c_l \sum_{t=1}^{T}\mathbb{E}[\|Y_t - \widehat Y_t \|] \notag \\
        \leq & \, \ell(Y_{{0}}) + c_l  D\,.
        \end{align}
    Note that with
    \begin{align*}
        \widehat U_t\widehat U_t^\top(I - P(Y_{t-1})) \nabla\ell(Y_{t-1}) =\,& \widehat U_t\widehat U_t^\top\big(\nabla\ell(Y_{t-1}) - U_{t-1}U_{t-1}^\top \nabla\ell(Y_{t-1})\\
        \,&+ (U_{t-1}U_{t-1}^\top-I) \nabla\ell(Y_{t-1})V_{t-1}V_{t-1}^\top\big)\\
        =\,& \widehat U_t\widehat U_t^\top\left(I - U_{t-1}U_{t-1}^\top\right) \nabla\ell(Y_{t-1})\left(I -  V_{t-1}V_{t-1}^\top\right) = 0\,
    \end{align*}
    and $\widehat U_t\widehat U_t^\top P(Y_{t-1}) \nabla\ell(Y_{t-1}) = P(Y_{t-1}) \nabla\ell(Y_{t-1})$ we have 
    \begin{align*}
      \widehat U_t\widehat U_t^\top \nabla\ell(Y_{t-1}) =\,& \widehat U_t\widehat U_t^\top P(Y_{t-1}) \nabla\ell(Y_{t-1})+\widehat U_t\widehat U_t^\top (I - P(Y_{t-1})) \nabla\ell(Y_{t-1})\\
      =\,& P(Y_{t-1}) \nabla\ell(Y_{t-1})\,.
    \end{align*}
    Hence, \eqref{eq:conv1} becomes 
        \begin{align*}
        \sum_{t=1}^{T}h_t \left(1-\frac{c_l {h_t}}{2}\right)\mathbb{E}[\Vert P(Y_{t-1}) \nabla\ell(Y_{t-1}) \Vert^2] \leq \ell(Y_{{0}}) + c_l D\,.
    \end{align*}
    Using Assumption~\ref{ass:bound}, i.e., $\Vert P(Y_{t-1}) \nabla\ell(Y_{t-1}) \Vert \leq B$, when $T\rightarrow \infty$, the right-hand side remains bounded, implying that
    \begin{align*}
        \liminf_{T\rightarrow\infty} \mathbb{E}[\Vert P(Y_{T})\nabla\ell(Y_{T}) \Vert^2] = 0\,.
    \end{align*}
\end{proof}

\section{Numerical Experiments}\label{sec:numexp}

The performance of the DLRT Algorithm \ref{algo:aPSI} is demonstrated training artificial neural network on the MNIST dataset and fine-tuning a vision transformer pre-trained on ImageNet. The implementation, available in PyTorch (\href{https://github.com/ScSteffen/Publication-Augmented-Backward-Corrected-PSI-low-rank-training}{GitHub repository}), was executed on a computer system equipped with an AMD Ryzen™ 9 3900X Processor, 128 GB RAM, and an NVIDIA GeForce RTX 3090 GPU with 24 GB VRAM. The software environment included Python 3.11.7, PyTorch 2.2.0, and CUDA 11.8.

\subsection{MNIST}
For each experiment, five neural networks with the following architecture were trained: an input layer with 784 nodes, four hidden layers with 500 nodes each, and an output layer with 10 nodes. First, five fully connected (dense) networks were trained as a baseline. A learning rate of $h=0.00001$ was used to avoid instability during training. The average test accuracy for the five dense networks is 94.54 $\pm$ 0.16.

The experimental setup included the three variations of the PSI method: (a) the original PSI (\cref{sec_psi}), (b) the backward-corrected PSI (\cref{sec_bcpsi}), and (c) the augmented backward-corrected PSI (\cref{sec:abcPSI}) outlined in Algorithm \ref{algo:aPSI}. Each setup was tested using learning rates of 0.01 and 0.001. Fixed ranks for setups (a) and (b) were determined based on results from experiment (c), which employed truncation tolerances of $\tau\in\{0.005, 0.01, 0.02, 0.05, 0.1,0.2\}$.

The average test accuracies of each setup along with the number of parameters computed over five models, are summarized in Table \ref{tab:vit_lr01} and Table \ref{tab:vit_lr001}. Table \ref{tab:vit_lr01} shows the results for training runs with a learning rate of 0.01 and Table \ref{tab:vit_lr001} results with a learning rate of 0.001. For learning rate 0.01, the original PSI encountered training failures for one model with rank 28 and all five models with rank 33. Experiments using the backward-corrected PSI with the same learning rate exhibited instability, with 15 out of 35 models failing in total. For the augmented backward-corrected PSI at a learning rate of 0.01, one out of five models failed to train for tolerances of 0.02, 0.15, and 0.2. For configurations in which only one out of five models failed, an additional model was trained to ensure representative comparisons. Notably, no training failures occurred during these new training runs. Also no failures occurred for any setup using a learning rate of 0.001.

\begin{table}[t]
\centering
 \caption{Mean test accuracy (acc.) with standard deviation of five training runs using original PSI (PSI), backward-corrected PSI (abc-PSI), and augmented backward-corrected PSI (abc-PSI) on the MNIST data set using learning rate 0.01 and different tolerances, and ranks, respectively. The number of parameters is denoted in Millions, abbreviated by "M".  {It is apparent that the bc-PSI and PSI fails to train for a wide range of $\tau$, whereas the abc-PSI  not only trains successfully for all $\tau$, but also outperforms PSI and bc-PSI in lower compression regimes.}}
\label{tab:vit_lr01}
\resizebox{\textwidth}{!}
{%
\begin{tabular}{lcccccc}
\toprule
&\multicolumn{2}{c}{ abc-PSI (ours) } &\multicolumn{2}{c}{ PSI }  &\multicolumn{2}{c}{ bc-PSI }  \\
Tol [$\tau$] & \# Params &Acc [\%] &\# Params &Acc [\%] & \# Params &Acc [\%] \\
\midrule
0.200 & {0.04M} & 95.222 $\pm$ 0.336 & {0.04M} & 95.710 $\pm$ 0.132 & {0.04M} & 83.964 $\pm$ 21.789\\
0.150 & {0.05M} & 95.672 $\pm$ 0.627 & {0.05M} & 96.206 $\pm$ 0.200 & {0.05M} & -\\
0.100 & {0.07M} & 96.310 $\pm$ 0.365 & {0.07M} & 96.472 $\pm$ 0.100 & {0.07M} & -\\
0.050 & {0.09M} & 96.646 $\pm$ 0.061 & {0.09M} & 96.648 $\pm$ 0.068 & {0.09M} & -\\
0.020 & {0.11M} & 96.894 $\pm$ 0.158 & {0.11M} & 96.650 $\pm$ 0.171 & {0.11M} & -\\
0.010 & {0.12M} & 97.222 $\pm$ 0.119 & {0.12M} & 96.588 $\pm$ 0.062 & {0.12M} & 89.350 $\pm$ 10.366\\
0.005 & {0.16M} & \textbf{97.422 $\pm$ 0.862} & {0.16M} & - & {0.16M} & 90.416 $\pm$ 9.760 \\
\bottomrule
\end{tabular}
}
\end{table}

\begin{table}[t]
\centering
 \caption{Mean test accuracy (acc.) with standard deviation of five training runs using original PSI (PSI), backward-corrected PSI (abc-PSI), and augmented backward-corrected PSI (abc-PSI) on the MNIST data set using learning rate 0.001 and different tolerances, and ranks, respectively. The number of parameters is denoted in Millions, abbreviated by "M". 
  {With a smaller learning rate 0.001, PSI and bc-PSI are able to train the network, however the abc-PSI achieves the highest validation accuracy values.  }}
\label{tab:vit_lr001}
\resizebox{\textwidth}{!}
{%
\begin{tabular}{lcccccc}
\toprule
&\multicolumn{2}{c}{ abc-PSI (ours) } &\multicolumn{2}{c}{ PSI }  &\multicolumn{2}{c}{ bc-PSI }  \\
Tol [$\tau$] & \# Params &Acc [\%] &\# Params &Acc [\%] & \# Params &Acc [\%] \\
\midrule
0.200 & {0.04M} & 90.650 $\pm$ 0.378 & {0.04M} & 92.910 $\pm$ 0.381 & {0.04M} & 93.116 $\pm$ 0.626\\
0.150 & {0.05M} & 92.006 $\pm$ 0.549 & {0.05M} & 93.778 $\pm$ 0.424 & {0.05M} & 93.584 $\pm$ 0.512\\
0.100 & {0.06M} & 93.080 $\pm$ 0.217 & {0.06M} & 94.102 $\pm$ 0.206 & {0.06M} &93.870 $\pm$ 0.388\\
0.050 & {0.07M} & 93.936 $\pm$ 0.271 & {0.07M} & 94.506 $\pm$ 0.295 & {0.07M} & 94.864 $\pm$ 0.358\\
0.020 & {0.08M} & 94.300 $\pm$ 0.121 & {0.08M} & 94.552 $\pm$ 0.153 & {0.08M} & 94.826 $\pm$ 0.430\\
0.010 & {0.08M} & 94.556 $\pm$ 0.204 & {0.08M} & 94.760 $\pm$ 0.137 & {0.08M} & 95.136 $\pm$ 0.598\\
0.0005 & {0.12M} & 95.938 $\pm$ 0.121 & {0.11M} & 95.060 $\pm$ 0.277 & {0.11M} & 95.400 $\pm$ 0.292\\
0.0003 & {0.14M} & 96.664 $\pm$ 0.223 & {0.15M} & 94.442 $\pm$ 0.160 & {0.15M} & 95.654 $\pm$ 0.304\\
0.0002 & {0.17M} &\textbf{ 96.936 $\pm$ 0.115} & {0.17M} & 95.950 $\pm$ 0.280 & {0.17M} & 94.242 $\pm$ 0.439 \\
\bottomrule
\end{tabular}
}
\end{table}

To measure the parameter reduction achieved through the dynamic low-rank approximation method, the compression rate was calculated as
\begin{align*}
\textrm{compression rate} = \left(1 - \frac{\sum_l (i_l + o_l) \cdot r_l}{\sum_l i_l \cdot o_l}\right) \cdot 100
\end{align*}
where $i_l$ and $o_l$ denote the input and output dimensions of layer $l$, respectively, and $r_l$ representing its rank. Figure \ref{fig:testACC_compression} compares the compression rate with the mean test accuracy across all setups, excluding bc-PSI with a learning rate of 0.01 due to frequent training failures.
The figure reveals that setups trained with a learning rate of 0.01 generally outperform those with smaller learning rates. Furthermore, accuracy improves as compression decreases in all configurations except for the original PSI method. This discrepancy could be attributed to unstable training dynamics. I.e., for the original PSI, no training was successful at low compression rates, as all models failed when using a rank of 33.
For compression rates exceeding 91\%, the original PSI with a learning rate of 0.01 outperforms all other methods, achieving its peak accuracy of 96.65\% with a rank of 25. However, this method becomes unstable when dealing with larger parameter counts, causing most training runs to fail. Notably, only models trained with the abc-PSI achieve accuracies above 97\% while maintaining substantial compression above 86\%. Thus, the best performance for the MNIST dataset was observed in the setup employing abc-PSI with a tolerance of 0.005 and a learning rate of 0.01. This configuration achieved the highest average test accuracy (97.42\%) across five models, as well as the highest accuracy for a single model (97.65\%).

\begin{figure}[h]
\centering
\includegraphics[width=0.9\textwidth]{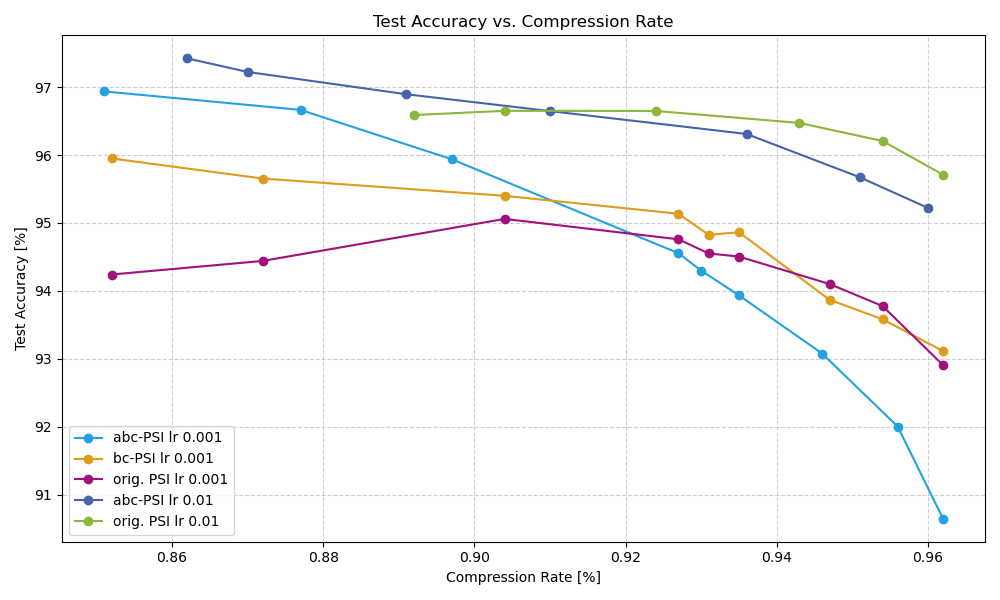}
\caption{Mean test accuracy of all experimental setups (PSI, bc-PSI, abc-PSI) trained on the MNIST dataset, plotted against their compression rates using learning rates of 0.01 and 0.001. Compression rates correspond to different rank selections (for fixed-rank settings) or varying tolerances (for rank-adaptive settings). Training with a learning rate of 0.01 was unstable for all backward-corrected PSI trainings and original PSI models with ranks $r > 28$, frequently leading to failed trainings; these cases are excluded from the graphic.} 
\label{fig:testACC_compression}
\end{figure}

\subsection{Vision Transformer fine-tuning for image classification}
We consider a pre-trained Vit-base-patch16-224 vision transformer and use the proposed augmented backward-corrected PSI to fine-tune the vision transformer on the smaller dataset. 
Fine-tuning means in this context, that an additive correction $Y$ is introduced for each pre-trained weight matrix $W_{\textup{pre}}$ of the neural network model. That is, each linear layer with input $x$ of the model, e.g. $Wx+b$, becomes $Wx + Yx +b$. 
The correction $Y$ is parametrized as $USV^\top$, thus the abc-PSI can readily be applied to fine-tune the pre-trained base model.  

We compare the proposed method to well known fine-tuning methods:
\begin{enumerate}
    \item Low-Rank Adaptation (LoRA)~\cite{hu2021lora}, which parametrizes $Y=AB^\top$, where $A,B\in\mathbb{R}^{n\times r}$ and $r$ is fixed. $A$ and $B$ are updated simultaneously by gradient descent. 
    \item AdaLoRA~\cite{zhang2023adalora}, which parametrizes $Y=USV^\top$, but in contrast to the proposed method, $U,S,$ and $V$ are updated by simultaneous gradient descent. $U$ and $V$ are regularized to be approximately orthogonal and a singular value truncation criterion on $S$ is used to mask or reactivate singular values and the corresponding basis functions.
    \item GeoLoRA~\cite{schotthöfer2024geolorageometricintegrationparameter}, a recently proposed rank-adaptive method for low-rank training and fine-tuning with convergence and optimality guarantees similar to the proposed method.
\end{enumerate}

We present in Table \ref{tab_vit} results for fine-tuning the vit-base-patch16-224 vision transformer, which is pre-trained on the ImageNet-1k-dataset. The pre-trained weights are downloaded from the torch-vision python package.
For all methods, we augment the key, query, and value matrices from attention layers as well as the three fully connected layers of each transformer block with a low-rank adapter. The biases of each layer are trainable. Additionally, the classifier is augmented with a low-rank adapter. The classifier layer is low-rank by construction, thus its rank is set to the number of classes.

We fine-tune the vision transformer on Cifar10 and Cifar100. Table \ref{tab_vit} shows the accuracies and number of parameters of the resulting models. Hyperparameter configurations used to produce these results are given in Table \ref{tab_vit_params}. 
The proposed abc-PSI achieves validation accuracies comparable to the methods in the literature, however, with significantly fewer parameters. The reported parameters constitute as $\sum_{l=1}^L m_lr_l + n_lr_l + r_l^2$, for $L$ low-rank adapter layers for all methods.
We remark that during training, the forward and gradient evaluation of the abc-PSI requires only the $K,V$ or the $L,U$ matrices at a time. Only in the truncation step, the $U,S,V$ matrices are required at the same time. This enables more sophisticated implementation strategies, to reduce the real memory footprint during the $K$ and $L$ step to $\sum_{l=1}^L m_lr_l + n_lr_l$. This is not possible in the rank adaptive literature methods AdaLoRA and GeoLoRA, that require $U,S,V$ and their gradients at all times. 

\begin{table}[t]
\centering
 \caption{Vit-base-patch16-224 fine-tuning on Cifar10, and Cifar100. We compare the number of parameters and the networks' accuracies of the abc-PSI to LoRA, AdaLoRA and GeoLoRA reporting the median of 5 runs. The number of parameters is denoted in Millions, abbreviated by "M". The abc-PSI achieves slightly higher validation accurcay for Cifar10 with less parameters and for Cifar100 achieves similar accuracy with slightly lower number of trainable parameters. }
\label{tab_vit}
{%
\begin{tabular}{lccccc}
\toprule
Method &\multicolumn{2}{c}{ Cifar 10 [\%]} &\multicolumn{2}{c}{ Cifar 100 [\%]}\\
& \# Params &Acc [\%] &\# Params &Acc [\%]\\
\midrule
{LoRA}   &{0.47M (r=3)}& {98.47} &{ 0.47M (r=3)}&{ 91.47}  \\
AdaLoRA    & {0.47M} & 98.51 & 0.45M & 91.44 \\
GeoLoRA & {0.47M} & {98.55} &{0.35M}&\textbf{91.63}\\
abc-PSI & \textbf{0.34M} & \textbf{98.57} & \textbf{0.34M} & 90.93 \\
\bottomrule
\end{tabular}
}
\end{table}

\begin{table}[t]
\centering
 \caption{Hyper-parameter setup for fine-tuning vision transformer with abc-PSI. }
\label{tab_vit_params}
{
\begin{tabular}{lccccc}
\hline Dataset & Learning Rate & Batch Size & \# Epochs & $\tau$ & inital rank \\
\hline Cifar10 & $8 \times 10^{-4}$ & 256 & 5 & 0.15 & 32 \\
Cifar100 & $1 \times 10^{-3}$ & 256 & 5 & 0.1 & 32 \\
\hline
\end{tabular}
}
\end{table}

\subsection{Discussion}
This paper introduces the augmented backward-corrected PSI (abc-PSI) method for robust and rank-adaptive low-rank training of neural networks. The abc-PSI is suitable for neural network compression during training and low-rank fine-tuning of pre-trained models. Compared to existing methods, it achieves competitive validation accuracy while providing greater network compression.

We have demonstrated that the proposed method is robust in the presence of small singular values, effectively reduces the training loss when used with stochastic gradient descent, and fulfills local convergence guarantees.

\section*{Acknowledgments}
All authors sincerely thank Martin Frank for his invaluable support, guidance, and the insightful discussions that helped shape this work. AW acknowledges the Helmholtz Information and Data Science Academy and the Norwegian Artificial Intelligence Research Consortium for funding her research visit at the Norwegian University of Life Sciences as well as the Norwegian University of Life Sciences for hosting her stay, during which much of this work was conducted.

The authors have used ChatGPT, version v2, to edit and polish written text for spelling, grammar, or general style. All authors have carefully examined, and refined the content, taking complete accountability for the finalized version of this manuscript.

\bibliographystyle{siamplain}
\bibliography{references}

\section*{Author contribution statement (CRediT)}
\textbf{Jonas Kusch:} Conceptualization, Methodology, Formal analysis, Writing - Original Draft, Supervision. 
\textbf{Steffen Schotth\"ofer:} Methodology, Software, Benchmarking, Writing - Original Draft, Supervision.
\textbf{Alexandra Walter:} Methodology, Formal analysis, Software, Writing - Original Draft, Visualization

\appendix
\section{Proof of Lemma \ref{lemma:loss}}\label{app:lemma4}
In the following, we restate Lemma 5.2. of \cite{arsen_stachastic_gradient_descent} for the stochastic gradient.
\begin{proof}
We have for a general $Z :\mathbb{R}_+ \rightarrow \mathbb{R}^{m \times n}$
\begin{align*}
\frac{d}{d t} \ell(Z(t))=\langle\nabla \ell(Z(t)), \dot{Z}(t)\rangle \, .
\end{align*}
With the fundamental theorem of calculus,
\begin{align}
\ell(Z_1) \overset{}{\hfill =} & \,\ell(Z_2)+\int_0^1 \frac{d}{d t} \ell(Z_2+t(Z_1-Z_2)) d t \notag\\
= & \,\ell(Z_2)-\int_0^1\langle \nabla \ell(Z_2+t(Z_1-Z_2)), Z_1-Z_2\rangle d t \, . \label{eq:L2_main}
\end{align}
Then, with zero completion using $\pm \nabla \ell(Z_2)$ and pulling out the of $t$ independent term from the integral, \eqref{eq:L2_main} becomes
\begin{align*}
\ell(Z_1) = & \, \ell(Z_2)-\langle \nabla \ell(Z_2), Z_1-Z_2\rangle \\
& \, -\int_0^1\langle \nabla \ell(Z_2+t(Z_1-Z_2))-\nabla \ell(Z_2), Z_1-Z_2\rangle d t\,.
\end{align*}
Using the Cauchy-Schwarz inequality and Assumption \ref{ass:Lipschitz}, yields
\begin{align*}
- \int_0^1\langle \nabla \ell(Z_2+t(Z_1-Z_2))- & \nabla \ell(Z_2), Z_1-Z_2\rangle d t \, \\
\leq &\int_0^1 \| \nabla \ell(Z_2+t(Z_1-Z_2))-\nabla \ell(Z_2)\| \cdot \| Z_1-Z_2 \| d t \\
= & \, c_l \int_0^1 \| (Z_2+t(Y-Z_2)-Z_2)\| \cdot \| Z_1-Z_2 \| d t \\
= & \, c_l \int_0^1  t \| (Z_1-Z_2)\| ^2 d t \, .
\end{align*}
Hence,
\begin{align*}
\ell(Z_1) \leq \, \ell(Z_2)-\langle \nabla \ell(Z_2), Z_1-Z_2\rangle +\frac{c_l}{2}\|Z_1-Z_2\|^2 \, ,
\end{align*}
concluding the proof of the Lemma.
\end{proof}
\end{document}